\documentclass{article}
\usepackage{amsfonts}
\usepackage{graphicx} % Required for inserting images
\usepackage{tikz-cd}
\title{\textbf{Calabi--Yau Hyperbolicity}}
\author{ABDELOUAHAB KHELIFATI}
\date{}
\usepackage{graphicx} 
\usepackage{booktabs}
\usepackage[english]{babel}
\usepackage{amsmath}
\usepackage{amssymb}
\usepackage{amsfonts}
\usepackage{amsthm}
\usepackage{mathrsfs}
\usepackage{amscd}
\usepackage{enumitem}
\usepackage{extpfeil}
\usepackage{extarrows}
\usepackage{musicography}
\usepackage[utf8]{inputenc}
\usepackage[T1]{fontenc}
\usepackage[parfill]{parskip}
\usepackage{hyperref}
\hypersetup{colorlinks=true, linkcolor=blue, citecolor=red, urlcolor=blue
}
\usepackage{lipsum}
\usepackage{mathtools}
\newcommand{\longdashrightarrow}{%
  \mathrel{\raisebox{0.2ex}{%
    \scalebox{1.5}[1]{$\dashrightarrow$}}}}
\usepackage[all]{xy}

\usepackage{makeidx}

\renewcommand{\thesubsection}{\arabic{section}.\arabic{subsection}}
\usepackage{fullpage}
\usepackage{amssymb, amsmath, amsthm, amscd}
\usepackage{titlesec}
\titleformat{\subsection}
  {\normalfont\Large\bfseries}
  {\thesubsection}{1em}{}
\numberwithin{equation}{section}

\theoremstyle{plain}
\newtheorem{theorem}{Theorem}[section]
\newtheorem{dfn}[theorem]{Definition}
\newtheorem{cl}[theorem]{Claim}
\newtheorem{thm}[theorem]{Theorem}
\newtheorem*{thmA}{Theorem A}
\newtheorem*{thmB}{Theorem B}
\newtheorem*{thmC}{Theorem C}
\newtheorem*{thmD}{Theorem D}
\newtheorem*{thmE}{Theorem E}
\newtheorem{prop}[theorem]{Proposition}
\newtheorem{lem}[theorem]{Lemma}
\newtheorem{cor}[theorem]{Corollary}
\newtheorem{conj}[theorem]{Conjecture}
\theoremstyle{definition}

\newtheorem{qst}[theorem]{Question}
\newtheorem{rmk}[theorem]{Remark}
\newtheorem{ex}[theorem]{Examples}
\newcommand{\B}{\mathbb{B}}
\newcommand{\C}{\mathbb{C}}

\newcommand{\T}{\mathbb{T}}
\newcommand{\Hyp}{\mathbb{H}}
\newcommand{\CP}{\mathbb{P}}
\newcommand{\R}{\mathbb{R}}
\newcommand{\Sp}{\mathbb{S}}
\newcommand{\N}{\mathbb{N}}
\newcommand{\Z}{\mathbb{Z}}

\begin{document}

\maketitle

\vspace{5ex}

\large{\textbf{Abstract.}} The goal of this paper is to introduce a new notion of hyperbolicity for compact complex manifolds with trivial canonical bundle, called \textit{Calabi--Yau hyperbolicity}. We investigate its fundamental properties and establish a characterization in terms of the fundamental group. In particular, we show that no compact Kähler Calabi--Yau manifold is Calabi--Yau hyperbolic, highlighting the genuinely non-Kähler nature of this notion. We also investigate different techniques for constructing examples of such manifolds. Finally, we examine the relationship between Calabi--Yau hyperbolicity and Kobayashi hyperbolicity in the examples considered, highlighting obstructions to the latter in the Calabi--Yau setting.
\section{Introduction}
\hspace{2pt} One of the first systematic notions of hyperbolicity for complex manifolds is the one introduced by S. Kobayashi in 1967 through his intrinsic pseudodistance attached to the complex manifold in question~\cite{kobayashi1967intrinsic}. A complex manifold is called \textbf{Kobayashi hyperbolic} when this pseudodistance is a genuine distance. In particular, Kobayashi hyperbolicity yields the absence of nonconstant entire curves. By Brody's theorem~\cite[Theorem 4.1]{brody1978compact}, the converse is also true for \textbf{compact} complex manifolds.

\hspace{2pt} The study of hyperbolicity has, however, a much longer history, going back to the small Picard theorem and to the hyperbolicity of compact Riemann surfaces of genus $g\geq2$, as follows from Liouville's theorem and the uniformization theorem. These classical examples naturally led to the study of higher-dimensional compact complex manifolds with positive canonical bundle. In the projective setting, this connection is reflected in the Kobayashi conjecture~\cite[Conjecture 7.2]{lang1986hyperbolic}:
\begin{conj}$($\textbf{Kobayashi ’70}$)$
Let $X$ be a compact Kähler manifold which is Kobayashi hyperbolic. Then $K_X$ is ample.
\label{1}
\end{conj}
\hspace{2pt} In particular, M. Gromov's notion of \textbf{Kähler hyperbolicity}~\cite[Definition 0.3]{gromov1991kahler} provides a very special class of compact Kähler Kobayashi hyperbolic manifolds for which the Kobayashi conjecture holds: the canonical bundle of a compact Kähler hyperbolic manifold is ample~\cite[Theorem 2.11]{chen2018compact}. This notion is defined by the existence of a Kähler form whose pullback to the universal cover admits a bounded primitive with respect to the lifted Kähler metric. This condition has strong consequences for the geometry and the topology of the manifold. In particular, the universal cover of a compact Kähler hyperbolic manifold is \textbf{open at infinity}~\cite[Definition II.5]{sullivan1976cycles}.

\hspace{2pt} In this paper, we explore a different direction by asking whether a compact complex manifold with trivial canonical bundle can nevertheless exhibit hyperbolic behavior at the level of the large-scale geometry of its universal cover. More precisely, we introduce a notion of \textbf{Calabi--Yau hyperbolicity} based on the existence of a bounded primitive for the volume form induced by a holomorphic volume form on the universal cover.

\hspace{2pt} Our approach thus seeks to understand which geometric and topological features can give rise to hyperbolic behavior even when the canonical bundle is trivial. A recent preprint by OpenAI~\cite{OAI-Kobayashi-2026} claims to establish the Kobayashi conjecture for compact Kähler manifolds in full generality. This would imply that no compact Kähler manifold with trivial canonical bundle is Kobayashi hyperbolic. This motivates the study of alternative notions of hyperbolicity that remain meaningful when the canonical bundle is trivial, particularly in the non-Kähler setting.

\hspace{2pt} A related perspective appears in the recent work of P. Li~\cite{li2026topological}. Among other results, P. Li proves that the fundamental group of a Calabi--Yau manifold cannot be isomorphic to that of a Kähler hyperbolic manifold~\cite[Theorem 2.8]{li2026topological}. This result illustrates how the topology of Calabi--Yau manifolds can obstruct certain forms of geometric hyperbolicity. Our approach is different: rather than requiring a Kähler structure, we study a condition on the holomorphic volume form that captures large-scale hyperbolic behavior and allows for genuinely non-Kähler examples.

\hspace{2pt} We first establish basic properties of Calabi--Yau hyperbolicity, including a characterization in terms of the fundamental group. In particular, we show:
\begin{thmA}[\textbf{Characterization of Calabi--Yau hyperbolic manifolds}] Let $X$ be a compact complex manifold with $K_X\simeq\mathcal{O}_X$. Then $X$ is Calabi--Yau hyperbolic if and only if $\pi_1(X)$ is non-amenable. In particular, no compact Kähler Calabi--Yau manifold is Calabi--Yau hyperbolic.
\end{thmA}
\hspace{2pt} Thus, although the definition is inspired by Kähler hyperbolicity, the resulting notion has a different geometric scope and leads naturally to the study of non-Kähler manifolds.

\hspace{2pt} The preceding discussion leads to the following result for Calabi--Yau manifolds in Fujiki class~$\mathcal{C}$, which follows from the recent preprint~\cite{OpenAIAbelianitySpecialKahler2026} claiming a proof of \textbf{Campana's abelianity conjecture}~\cite[Conjecture 7.1]{campana2004orbifolds} in all dimensions.
\begin{thmB}
Let $X$ be a compact Calabi--Yau manifold in Fujiki class~$\mathcal{C}$. Then $\pi_1(X)$ is virtually abelian, and hence amenable. In particular, $X$ is not Calabi--Yau hyperbolic.
\end{thmB}
\hspace{2pt} This leads us to ask whether a similar obstruction to Calabi--Yau hyperbolicity holds for the broader class of compact Calabi--Yau manifolds satisfying the $\partial\bar\partial$-property. More precisely, we formulate the following conjecture.
\begin{conj}[\textbf{Fundamental groups of Calabi--Yau $\partial\bar\partial$-manifolds}]
Let $X$ be a compact Calabi--Yau $\partial\bar\partial$-manifold. Then $\pi_1(X)$ is amenable. In particular, $X$ is not Calabi--Yau hyperbolic.
\end{conj}
\hspace{2pt} In Section \S\ref{3.2}, we present the following examples of Calabi--Yau hyperbolic manifolds:
\begin{enumerate}
    \item[\textbf{(\textit{a})}] Compact quotients $X=G/\Gamma$, where $G$ is a simply connected semisimple complex Lie group and $\Gamma\subset G$ is a cocompact lattice (called \textbf{Yachou manifolds}~\cite{yachou1998varietes}).
    \item[\textbf{(\textit{b})}] Generalized Calabi--Gray manifolds $X$ associated with a hyperkähler manifold $M$ and a pair $(Y,\varphi)$, as constructed by T. Fei~\cite{fei2016construction}. In this case, $X$ is Calabi--Yau hyperbolic if and only if $\pi_1(Y)$ is non-amenable.
\end{enumerate}
\hspace{2pt} We also investigate a geometric consequence of Calabi--Yau hyperbolicity in terms of higher-dimensional holomorphic maps. Motivated by the classical implication from Kähler hyperbolicity to Kobayashi hyperbolicity, we consider the notion of \textbf{\textit{n}-cyclic hyperbolicity}, defined by the nonexistence of non-degenerate holomorphic maps from $\C^n$ into $X$ satisfying a suitable \textit{subexponential growth} condition~\cite[Definition~2.3]{marouani2023balanced}.
\begin{thmC}
Every compact complex manifold $X$ with non-amenable fundamental group is $n$-cyclically hyperbolic, where $n=\dim_\C X$. In particular, every Calabi--Yau hyperbolic manifold is $n$-cyclically hyperbolic. \label{C}
\end{thmC}
\hspace{2pt} Beyond the mentioned examples, we discuss methods for constructing new Calabi--Yau hyperbolic manifolds from known ones. These constructions are gathered in the following theorem.
\begin{thmD}[\textbf{Construction of Calabi--Yau hyperbolic manifolds}] \leavevmode
\begin{enumerate}
    \item If $X$ is a Calabi--Yau hyperbolic manifold and $Y$ is a compact Calabi--Yau manifold, then $X\times Y$ is Calabi--Yau hyperbolic.
    \item If $X$ and $Y$ are two compact Calabi--Yau manifolds, and $f:X\longrightarrow Y$ is a finite étale cover, then $X$ is Calabi--Yau hyperbolic if and only if $Y$ is Calabi--Yau hyperbolic.
    \item Let $X$ be a Calabi--Yau hyperbolic manifold, and let $\mu:X\longrightarrow Y$ be a crepant modification. Then $Y$ is Calabi--Yau hyperbolic.
    \item The following families of deformations consist of Calabi--Yau hyperbolic manifolds:
    \begin{enumerate}
        \item[(a)] Let $X$ be a generalized Calabi--Gray manifold associated with a hyperkähler manifold $M$ and a pair $(Y,\varphi)$. If $\{M_t\}_{t\in\mathbb{B}_\varepsilon}$ is a holomorphic family of small deformations of $M$, then there exists a holomorphic family $\{X_t\}_{t\in\mathbb{B}_\varepsilon}$ of generalized Calabi--Gray manifolds associated with $\{M_t\}_{t\in\mathbb{B}_\varepsilon}$ and the fixed pair $(Y,\varphi)$, with $X_0=X$. In particular, each $X_t$ is Calabi--Yau hyperbolic.
        \item[(b)] Let $X=\mathrm{SL}_2(\mathbb{C})/\Gamma$, where $\Gamma\subset\mathrm{SL}_2(\mathbb{C})$ is a torsion-free cocompact lattice with $b_1(\Gamma)>0$. The family of deformations $\{X_\rho\}_{\rho\sim\rho_0}$ arising from Ghys's construction defines a non-trivial family of Calabi--Yau hyperbolic manifolds.
    \end{enumerate}
    \item Let $Y$ be a compact CY-hyperbolic manifold, and let $p:X\longrightarrow Y$ be a holomorphic principal elliptic bundle. Then $X$ is CY-hyperbolic.
\end{enumerate} \label{D}
\end{thmD}
\hspace{2pt} We conclude by discussing obstructions to Kobayashi hyperbolicity for Calabi--Yau manifolds in arbitrary dimension, as well as for compact complex surfaces. The following theorem summarizes these obstructions and their implications for the examples and constructions considered in this paper.
\begin{thmE}[\textbf{Obstructions to Kobayashi hyperbolicity}] \leavevmode
\begin{enumerate}
    \item Assuming the validity of the claimed proof of the Global Spherical Shell (GSS) conjecture~\cite{OAI-GSS-2026}, every compact Kobayashi hyperbolic complex surface is Kähler.
    \item The examples of Calabi--Yau hyperbolic manifolds constructed from Yachou manifolds and generalized Calabi--Gray manifolds, as well as the further manifolds obtained from them by the constructions considered in \hyperref[D]{Theorem~D}, are not Kobayashi hyperbolic.
\end{enumerate}
\end{thmE}
\section*{Acknowledgments}
\hspace{2pt} I would like to thank Dan Popovici for taking the time to look at this paper, for drawing my attention to his recent joint works with H. Kasuya and L. Ugarte \cite{kasuya2025higher,kasuya2025properties}, and for suggesting that I investigate the relationship between Calabi--Yau hyperbolicity and its geometric counterpart, namely $n$-cyclic hyperbolicity, as addressed in \hyperref[C]{Theorem~C}.
\section{Classical hyperbolicity notions}
\hspace{2pt}Throughout this paper, $X$ will be a compact connected complex manifold of dimension $\dim_\C X=n$ unless otherwise stated, and $\pi:\widetilde{X}\longrightarrow X$ its universal cover. In particular, if $\omega$ is some Hermitian metric on $X$, then $\widetilde{\omega}=\pi^*\omega$ is a \textit{complete} Hermitian metric on $\widetilde{X}$. We begin by recalling the classical notions of hyperbolicity that motivate the definition studied in this paper, together with some of their important properties.
\begin{dfn}
Let $X$ be a (not necessarily compact) complex manifold. We say that $X$ is \textbf{Kobayashi hyperbolic} if the \textbf{Kobayashi pseudo-distance} $($defined in~\cite{kobayashi1967intrinsic}$)$ is a genuine distance on $X$.
\end{dfn}
R. Brody  gave a characterization of Kobayashi hyperbolicity in the compact setting~\cite{brody1978compact}. Namely:
\begin{thm} Let $X$ be a compact complex manifold. Then:
\begin{center}
    X is Kobayashi hyperbolic $\Longleftrightarrow$ $X$ contains no nonconstant entire curves
\end{center} \label{4}
\end{thm}
The latter property has come to be called \textbf{Brody hyperbolicity}. 

\hspace{2pt}A distinguished class of compact Kähler Kobayashi hyperbolic manifolds is given by the \textit{Kähler hyperbolic manifolds}, introduced by M. Gromov in one of his seminal papers~\cite{gromov1991kahler}:
\begin{dfn}
    A compact Kähler manifold $X$ is called \textbf{Kähler hyperbolic} if it admits an \textbf{$\widetilde{d}($bounded$)$} Kähler form $\omega$, i.e. $\widetilde{\omega}:=\pi^*\omega=d\zeta$, where $\zeta$ is a $L^\infty_{\widetilde{\omega}}$-bounded 1-form on $\widetilde{X}$.
\end{dfn}

\hspace{2pt}The bounded-primitive condition in Gromov's definition of Kähler hyperbolicity has motivated several extensions of hyperbolicity to compact complex manifolds endowed with special Hermitian structures, including \textbf{balanced hyperbolicity}~\cite{marouani2023balanced}, \textbf{SKT hyperbolicity} and \textbf{Gauduchon hyperbolicity}~\cite{marouani2023skt} (by replacing the $\widetilde d$-boundedness condition with a boundedness condition formulated in terms of Aeppli cohomology, which is called $(\widetilde{\partial+\bar\partial})$\textbf{-boundedness}, \textbf{$p$-Kähler hyperbolicity}~\cite{haggui2023compact}, \textbf{sG hyperbolicity} and \textbf{$p$-HS hyperbolicity}~\cite{ma2024strongly}, as well as \textbf{HS hyperbolicity} and \textbf{$p$-SKT hyperbolicity}~\cite{khelifati2026deformations}. Also recently, partial variants such as \textbf{partial $p$-Kähler hyperbolicity} have been introduced~\cite{kasuya2025partially}.

\hspace{2pt}The bounded-primitive condition in the definition of Kähler hyperbolicity is also closely related to the linear isoperimetric geometry of the universal cover. Let $(M,g)$ be a complete oriented Riemannian manifold. Its \textbf{\textit{modified} Cheeger constant} is given by:
\begin{equation*}
    h_g(M):=\inf_{D\Subset M}\dfrac{\operatorname{Area}_g(\partial D)}{\operatorname{Vol}_g(D)}
\end{equation*}
where the infimum is taken over all relatively compact smooth domains $D\Subset M$. This is the natural non-compact version of the isoperimetric constant introduced by J. Cheeger in~\cite[Definition 1]{cheeger1970lower} for compact Riemannian manifolds.
\begin{rmk}
The Cheeger constant $h(M,g)$ depends on the choice of the Riemannian metric $g$. For instance, if $g'=\lambda^2g$ for some $\lambda>0$, then:
\begin{equation*}
\operatorname{Area}_{g'}(\partial D)=\lambda^{m-1}\operatorname{Area}_{g}(\partial D),\qquad\operatorname{Vol}_{g'}(D)=\lambda^m\operatorname{Vol}_{g}(D),
\end{equation*}
where $m=\dim_{\R}M$. Consequently: $h_{g'}(M)=\lambda^{-1}h_g(M)$. However, when $M=\widetilde{X}$ is the universal cover of a compact Riemannian manifold $X$, the positivity of $h_{\widetilde{g}}(M)$ is independent of the choice of the Riemannian metric $g$ on $X$, since any two Riemannian metrics on $X$ are uniformly equivalent and so are their lifts to $M$. \label{2}
\end{rmk}

\hspace{2pt} When $M$ has \textbf{bounded geometry}, i.e. the sectional curvature is bounded and the injectivity radius is bounded away from zero, we have the following:
\begin{dfn}$($\cite[Definition II.5, Proposition II.6]{sullivan1976cycles},~\cite[\S4.1]{gromov1981hyperbolic},~\cite[Corollary 1.2]{sikorav2001growth}$)$
    Let $(M^m,g)$ be a complete oriented Riemannian manifold that has bounded geometry. Then $M$ is \textbf{open at infinity} if one of the following equivalent conditions holds:
    \begin{enumerate}
        \item The volume form $dV_g$ is $d($bounded$)$, i.e. $dV_g=d\tau$, where $\tau$ is an $L_g^\infty$-bounded $(m-1)$-form on $M$.
        \item $h_g(M)>0$, i.e. The following \textbf{linear isoperimetric inequality} holds for some constant $C>0:$
        \begin{equation*}
            \operatorname{Vol}_g(D)\leq C\operatorname{Area}_g(\partial D)
        \end{equation*}
        for any relatively compact smooth domain $D\Subset M$ $($One can take $C=h_g(M)^{-1})$.
    \end{enumerate}
\end{dfn}
It is easy to see that the universal cover $\widetilde{X}$ of a Kähler hyperbolic manifold ($X,\omega$) is open at infinity, as $\widetilde{\omega}^n=d(\zeta\wedge\widetilde{\omega}^{n-1})$ and:
\[
\lVert\zeta\wedge\widetilde{\omega}^{n-1}\rVert_{L^\infty_{\widetilde{\omega}}(\widetilde{X})}\leq\lVert\zeta\rVert_{L^\infty_{\widetilde{\omega}}(\widetilde{X})}\lVert\omega\rVert^{n-1}_{L^\infty_{\omega}(X)}<+\infty.
\]
\hspace{2pt}When $M$ is the universal cover of a closed Riemannian manifold $X$, the openness at infinity property of $M$ is also related to the fundamental group $\pi_1(X)$ by a theorem of R. Brooks (see also~\cite[\S4.2]{gromov1981hyperbolic}):
\begin{thm}\cite[Theorem 1]{brooks1981fundamental}
    Let $X$ be a closed Riemannian manifold and $M=\widetilde{X}$ its universal cover. Then $M$ is open at infinity if and only if $\pi_1(X)$ is non-amenable. \label{8}
\end{thm}
Recall that a finitely generated group $G$ is called \textbf{amenable} if it admits an invariant mean~\cite{vonneumann1929allgemeinen}. Equivalently, by Følner's criterion~\cite{folner1955groups}, there exists
a sequence $(F_k)_{k\geq1}$ of non-empty finite subsets of $G$ such that:
\[
\lim_{k\longrightarrow\infty}
\frac{|gF_k\triangle F_k|}{|F_k|}
=0
\qquad\text{for every }g\in G.
\]
A group which is not amenable is called \textbf{non-amenable}. 
\begin{rmk}
A result of G.M. Adel'son-Vel'skii and Y.A. Shreider in~\cite{adelʹson1957banach} asserts that every finitely generated group of subexponential growth is amenable. Consequently, every non-amenable finitely generated group has exponential growth. \label{3}
\end{rmk}

\hspace{2pt} The various notions of hyperbolicity recalled above are closely related to the positivity properties of the canonical bundle. Recall that the \textbf{canonical bundle} of $X$ is the holomorphic line bundle of $(n,0)$-forms on $X$:
\[
K_X:=\Lambda^{n,0}T^*X=\operatorname{det}(\Lambda^{1,0}T^*X)=-\operatorname{det}(T^{1,0}X)
\]
\hspace{2pt} In the compact Kähler setting, the Kobayashi conjecture~\ref{1} predicts that Kobayashi hyperbolicity forces $K_X$ to be ample. This is known, in particular, for Kähler hyperbolic manifolds, for which the ampleness of the canonical bundle was established by B.-L. Chen and X. Yang~\cite[Theorem 2.11]{chen2018compact}. More recently, a proof of the Kobayashi conjecture has been claimed in~\cite{OAI-Kobayashi-2026}.

\hspace{2pt} At the opposite end, \textit{Fano manifolds} provide a basic class of examples which cannot be hyperbolic, either in the Kobayashi sense or in the large-scale geometric sense considered above. Recall that a compact complex manifold $X$ is called \textbf{Fano} if $-K_X$ is ample. Fano manifolds are rationally connected by the independent results of F. Campana~\cite{campana1992connexite} and Kollár--Miyaoka--Mori~\cite{kollar1992rationally}, and rationally connected manifolds are simply connected. In particular, $\pi_1(X)$ is amenable. Thus, Fano manifolds are also excluded from the large-scale hyperbolic picture discussed above. Moreover, Fano manifolds contain nonconstant rational curves and hence nonconstant entire curves, so Brody's theorem~\ref{4} implies that they are Kobayashi non-hyperbolic.

\hspace{2pt} More generally, if $X$ is a compact Kähler manifold such that $-K_X$ is nef, then a theorem of M. Păun shows that $\pi_1(X)$ has polynomial growth~\cite[Theorem 1]{paun1998varietes,paun1997groupe}. Thus, $\pi_1(X)$ is amenable by Remark~\ref{3}. The case $K_X\simeq\mathcal{O}_X$ is a particular instance of this situation. Related topological obstructions have recently been studied by P. Li in~\cite{li2026topological}. In particular, P. Li proves that a Calabi--Yau manifold of even complex dimension $n\geq 4$ is not homeomorphic to a weak Fano manifold, provided that $b_2=1$ when $n\geq 6$ \cite[Theorem 2.1]{li2026topological}. Moreover, he shows that the fundamental group of a Calabi--Yau manifold cannot be isomorphic to that of a Kähler hyperbolic manifold \cite[Theorem 2.8]{li2026topological}. These results further illustrate the distinction between Calabi--Yau geometry and geometric settings characterized by positivity or Kähler hyperbolicity.

\hspace{2pt} This naturally leads to the question of whether the situation can be different in the non-Kähler setting: can a compact complex manifold with trivial canonical bundle nevertheless have a universal cover exhibiting genuinely hyperbolic large-scale geometry?

\section{Calabi--Yau hyperbolicity}
\hspace{2pt}We begin by recalling some classical examples of compact complex manifolds with trivial canonical bundle, distinguishing between the Kähler and non-Kähler settings. These examples will provide the geometric context for our study of \textit{Calabi--Yau hyperbolicity}.
\subsection{Calabi--Yau manifolds}
\hspace{2pt}Throughout this paper, we use the terminology \textbf{Calabi--Yau manifold} in the broad sense, without imposing any Kähler, projectivity, or simply-connectedness assumption.
\begin{dfn} A compact complex manifold $X$ is said to be a \textbf{Calabi--Yau manifold} if one of the following equivalent conditions holds:
\begin{enumerate}[label=(\roman*)]
    \item The canonical bundle of $X$ is trivial, i.e. $K_X\simeq\mathcal{O}_X$.
    \item $\exists\Omega\in H^0(X,K_X)\simeq H^{n,0}(X,\C)$ such that $\Omega(x)\neq 0$, $\forall x\in X$.
    \item $\exists\Omega\in \mathscr{C}_{n,0}^\infty(X)\cap\operatorname{Ker}\bar\partial$ such that $\Omega(x)\neq 0$, $\forall x\in X$.
\end{enumerate} \label{7}
\end{dfn}
In particular, if $K_X\simeq\mathcal{O}_X$, then $h^{n,0}(X)=1$. Indeed, since $K_X$ is trivial, it admits a nowhere-vanishing holomorphic section $\Omega\in H^0(X,K_X)$ (called a \textbf{Calabi--Yau form}). Every other holomorphic section of $K_X$ can then be written as $f\Omega$, for some holomorphic function $f$. Since $X$ is compact and connected, $f$ must be constant on $X$. Thus $H^0(X,K_X)=\C[\Omega]$.
\begin{ex} There are several important classes of Calabi--Yau manifolds in this sense, arising from rather different geometric settings. We briefly recall some of them:
\begin{enumerate}
      \item \textbf{Projective Calabi--Yau manifolds.} Let $X\subset\CP^{n+1}$ be a smooth hypersurface of degree $d$. By the adjunction formula, its canonical bundle is given by
      \[
      K_X\simeq\mathcal{O}_X(d-n-2).
      \]
      In particular, if $d=n+2$, then $K_X\simeq\mathcal{O}_X$, and hence $X$ is a Calabi--Yau manifold. 
      \item \textbf{Kähler Calabi--Yau manifolds.} Several important classes arise in the Kähler setting:      \begin{enumerate}[label=(\textit{\alph*})] 
            \item \textbf{Complex tori.} A complex torus $\T^n=\C^n/\Lambda$ (where $\Lambda\subset\C^n$ is a lattice) is Kähler and has trivial canonical bundle.
            \item \textbf{Hyperkähler manifolds.} A compact \textit{hyperkähler} manifold $X$ is a Kähler manifold admitting a \textit{holomorphic symplectic form}, i.e. an everywhere non-degenerate $2$-form $\sigma\in H^0(X,\Omega_X^2)$ such that $\bar\partial\sigma=0$. In particular, $\dim_{\C}X=2m$ for some $m\geq 1$, and $\sigma^m$ is a nowhere-vanishing holomorphic volume form, thus $K_X\simeq\mathcal{O}_X$. By Yau's solution of the Calabi conjecture (Theorem~\ref{0}) and the Bochner principle, $X$ is hyperkähler if and only if the underlying smooth manifold admits a Riemannian metric $g$ and three complex structures $I,J,K$ satisfying the quaternionic relations: $I^2=J^2=K^2=IJK=-\operatorname{Id}$, such that $g$ is Kähler with respect to each of them. The associated complex structures are parametrized by a sphere $\Sp^2\simeq\CP^1$ (see~\cite{beauville1983varietes}). The corresponding complex structure on $X\times\CP^1$ defines the \textit{twistor space} of $X$, denoted by $\mathcal{Z}_X$, which comes with a natural projection $\pi:\mathcal{Z}_X\longrightarrow\CP^1$.
            
            A compact hyperkähler manifold $X$ is \textit{irreducible}, or called \textit{irreducible holomorphic symplectic}, if it is simply connected and admits a holomorphic symplectic form $\sigma$ satisfying $H^0(X,\Omega_X^2)=\C\sigma$. In terms of Riemannian holonomy, this corresponds to a hyperkähler metric whose holonomy group is exactly $\operatorname{Sp}(m)$. \textit{K3 surfaces} are the basic examples in complex dimension two. In higher dimensions, examples include \textit{Hilbert schemes} of points on K3 surfaces and \textit{generalized Kummer varieties}.
            \item \textbf{Strict Calabi--Yau manifolds.} In the restrictive sense, a compact simply connected Kähler manifold $Y$ of dimension $n$ is called a \textit{strict Calabi--Yau manifold} if $K_Y\simeq\mathcal{O}_Y$ and $H^0(Y,\Omega_Y^p)=0$ for all $0<p<n$. In complex dimension two, these are just K3 surfaces.
       \end{enumerate}
       \hspace{2pt}We recall a fundamental consequence of the triviality of the canonical bundle in the Kähler setting, which follows from S.-T. Yau's solution of the Calabi conjecture~\cite{yau1978calabi}:
       \begin{thm} Let $X$ be a compact Kähler manifold with $c_1(X)=0$ in $H^2(X,\R)$. Then every Kähler class on $X$ contains a unique Ricci-flat Kähler metric. \label{0}
       \end{thm}
       \hspace{2pt}In particular, every compact Kähler Calabi--Yau manifold admits a Ricci-flat Kähler metric. The structure of such manifolds is further described by the \textbf{Beauville--Bogomolov decomposition} theorem:
       \begin{thm}\cite{bogomolov1974decomposition,bogomolov1974kahler,kobayashi1981recent,beauville1983varietes} Let $X$ be a compact Ricci-flat Kähler manifold. Then there exists a finite étale cover $\widehat{X}\longrightarrow X$ such that:
       \[
       \widehat{X}\simeq T^q\times\prod_{j=1}^r Y_j^{n_j}\times\prod_{k=1}^s Z_k^{2m_k},
       \]
       where $T^q$ is a complex torus of dimension $q$, each $Y_j^{n_j}$ is a strict Calabi--Yau manifold of complex dimension $n_j$, and each $Z_k^{2m_k}$ is an irreducible hyperkähler manifold of complex dimension $2m_k$. \label{5}
       \end{thm}
       Using the exact sequence of fundamental groups associated with a fibration, we deduce from Theorem~\ref{5} that:
       \begin{cor}\cite[Corollary 4.2]{demailly2015structure} Let $X$ be a Kähler Calabi--Yau manifold. Then there is an exact sequence
       \[
       0\longrightarrow\Z^{2q}\longrightarrow\pi_1(X)\longrightarrow G\longrightarrow 0
       \]
       where $G$ is a ﬁnite group. Namely, $\pi_1(X)$ is almost abelian $($hence amenable$)$ and is an extension of a ﬁnite group $G$ by the normal subgroup $\pi_1(\widehat{X})\simeq\Z^{2q}$. \label{6}
       \end{cor}
       \item \textbf{Non-Kähler Calabi--Yau manifolds.} The class of compact non-Kähler Calabi--Yau manifolds contains a wide variety of geometries and topologies.
       \begin{enumerate}[label=(\textit{\alph*})]
       \item \textbf{Primary Kodaira surfaces.} A \textit{primary Kodaira surface} is a minimal compact non-Kähler complex surface of Kodaira dimension $0$ with trivial canonical bundle and first Betti number equal to 3. Its universal cover is biholomorphic to $\C^2$ (see, for instance,~\cite{kodaira1964structure,kodaira1966structure,barth2004compact}).
       \item \textbf{Compact complex parallelizable manifolds.} A compact complex manifold $X$ is called \textit{complex parallelizable} if its holomorphic tangent bundle is holomorphically trivial, i.e. $T_X\simeq\mathcal{O}_X^{\oplus n}$~\cite{wang1954complex}. In particular, $K_X=\det(T_X^*)\simeq\mathcal{O}_X$. By Wang's theorem~\cite[Theorem 1]{wang1954complex}, every compact complex parallelizable manifold can be expressed as a quotient $G/\Gamma$, where $G$ is a connected complex Lie group and $\Gamma\subset G$ is a discrete cocompact subgroup. Moreover, a compact complex parallelizable manifold is Kähler if and only if it is a complex torus~\cite[Corollary 2]{wang1954complex}. Two classical three-dimensional examples are particularly relevant:
       \begin{enumerate}[label=\textit{(\roman*)}]
       \item The \textbf{Iwasawa manifold.} Let $H_{\C}$ be the complex Heisenberg group
       \[
       H_{\C}=\left\{
       \begin{pmatrix}
       1 & z_1 & z_3\\
       0 & 1 & z_2\\
       0 & 0 & 1
       \end{pmatrix}
       \;\middle|\;
       z_1,z_2,z_3\in\C
       \right\},
       \]
       and let $\Gamma$ be its subgroup consisting of matrices whose entries $z_1,z_2,z_3$ belong to the Gaussian integers $\Z[i]$. The quotient $I:=H_{\C}/\Gamma$ is called the \textit{Iwasawa manifold}~\cite{nakamura1975complex}. It is a non-Kähler compact complex-parallelizable threefold, so it has trivial canonical bundle.
       \item \textbf{Compact quotients of $\mathrm{SL}_2(\C)$.} Let $\Gamma\subset \mathrm{SL}_2(\C)$ be a torsion-free cocompact lattice. Then $X=\mathrm{SL}_2(\C)/\Gamma$ is a compact non-Kähler complex-parallelizable threefold, hence it has trivial canonical bundle. $\mathrm{SL}_2(\C)$ admits a nowhere-vanishing holomorphic volume form invariant under left translations, which descends to the quotient.
       \end{enumerate}
       \item \textbf{Non-Kähler Calabi--Yau threefolds with trivial fundamental group.} In complex dimension three, R. Friedman~\cite{friedman1986simultaneous} and Lu--Tian~\cite{lu1994complex} constructed examples of compact complex manifolds with trivial canonical bundle on the underlying smooth manifolds $X_k=\mathop{\#}_k(\Sp^3\times \Sp^3)$, for integers $k\geq 2$ via \textit{conifold transitions}. These manifolds are simply connected and cannot be Kähler.
       \item \textbf{Examples arising from hyperbolic geometry.} J. Fine and D. Panov~\cite{finepanov2010hyperbolic} constructed further compact non-Kähler Calabi--Yau manifolds by exploiting hyperbolic geometry. Among their results, they obtained infinitely many distinct complex structures with trivial canonical bundle on $2(\Sp^3\times \Sp^3)\mathbin{\#}(\Sp^2\times\Sp^4)$.
       \item \textbf{Generalized Calabi--Gray manifolds.} Let $M$ be a compact hyperkähler manifold with $\dim_\C M=2m$. Given a compact complex manifold $Y$ and a non-constant holomorphic map $\varphi: Y\longrightarrow\CP^1$ satisfying $K_Y\simeq\varphi^*\mathcal{O}_{\CP^1}(2m)$, the associated \textit{generalized Calabi--Gray manifold} is the fiber product
       \[
       X:=Y\times_{\CP^1}\mathcal{Z}_M=\{(y,z)\in Y\times\mathcal{Z}_M\mid \varphi(y)=\pi(z)\}.
       \]
       As smooth manifolds, $X\underset{\mathscr{C^\infty}}{\cong}M\times Y$. Moreover, $K_X\simeq\mathcal{O}_X$~\cite{fei2016construction,fei2018generalized,fei2024remarks}.
       \item \textbf{Holomorphic \textit{p}-contact manifolds.} Recently, H. Kasuya, D. Popovici and L. Ugarte introduced the notion of a \textit{holomorphic $p$-contact structure} on a compact complex manifold $X$ with $\dim_\C X=n=2p+1$ \cite{kasuya2025higher}. This is a holomorphic $(p,0)$-form $\Gamma$ such that $\Gamma\wedge\partial\Gamma$ is nowhere vanishing on $X$. In particular, every holomorphic $p$-contact manifold is Calabi--Yau, since $\Omega_{\Gamma}:=\Gamma\wedge\partial\Gamma$ is a nowhere-vanishing holomorphic section of the canonical bundle $K_X$, hence $K_X\simeq\mathcal{O}_X$. The authors construct two classes of examples of such manifolds in arbitrary dimensions that arise as higher-dimensional generalizations of the classical three-dimensional Iwasawa manifold \cite[Proposition~3.2]{kasuya2025higher}.
       
       In the classical case $p=1$, five classes of compact complex threefolds carrying holomorphic contact structures are known: the Iwasawa manifold, the Nakamura manifolds, compact quotients of $\mathrm{SL}_2(\C)$, a non-complex-parallelisable nilmanifold, and a non-complex-parallelisable solvmanifold \cite[Section \S3]{kasuya2025properties}. Within the setting of compact quotients of six-dimensional Lie groups endowed with left-invariant complex structures, Propositions~3.2 and~3.4 of \cite{kasuya2025properties} establish the corresponding classification results for nilmanifolds and solvmanifolds, respectively. Relying on the classification of A. Otal and L. Ugarte, Observation~3.5 in \cite{kasuya2025properties} identifies the Lie algebra underlying $\mathrm{SL}_2(\C)$ as the unique unimodular non-solvable real Lie algebra of dimension six admitting a complex structure with a non-zero closed $(3,0)$-form \cite{otal2023six}.
       \end{enumerate}
       \hspace{2pt}The examples above show that compact Calabi--Yau manifolds can exhibit substantially different geometric and topological properties in the absence of the Kähler condition. 
\end{enumerate}
\end{ex}
\subsection{Calabi--Yau hyperbolicity: Definition and examples}\label{3.2}
\hspace{2pt} An interesting question concerning compact Calabi--Yau manifolds is whether they can exhibit any form of hyperbolicity. In the Kähler setting, S. Kobayashi conjectured that the Kobayashi pseudometric vanishes identically on every compact Calabi--Yau manifold~\cite[Problem~C.1]{kobayashi1976intrinsic}. This conjecture is known to hold for all K3 surfaces and for certain compact hyperkähler manifolds, as established by L. Kamenova, S. Lu, and M. Verbitsky~\cite{kamenova2014kobayashi}. A weaker form of the conjecture asserts that every compact Kähler Calabi--Yau manifold is Kobayashi non-hyperbolic. Using mirror symmetry, L. Kamenova and C. Vafa proved this weaker statement for Kähler Calabi--Yau manifolds whose mirror dual exists and is not Hodge degenerate~\cite{kamenova2020kobayashi}. 

\hspace{2pt} Motivated by this question, we introduce a different notion of hyperbolicity for compact Calabi--Yau manifolds. Unlike Kobayashi hyperbolicity, our definition involves a large-scale geometric condition on the universal cover of a compact Calabi--Yau manifold. This approach allows us to investigate geometric properties of compact Calabi--Yau manifolds beyond the Kähler setting.
\begin{dfn}
    Let $(X,\omega)$ be a compact Hermitian manifold with $K_X\simeq\mathcal{O}_X$, and equip its universal cover $\pi:\widetilde{X}\longrightarrow X$ with the lifted Hermitian metric $\widetilde{\omega}:=\pi^*\omega$. $X$ is \textbf{Calabi--Yau hyperbolic} $($or simply \textbf{CY-hyperbolic}$)$ if there exists a Calabi--Yau form $\Omega$ such that the associated volume form $\nu_\Omega:=c_ni^{n^2}\Omega\wedge\bar\Omega$ $($where $c_n>0)$ is $\widetilde{d}($bounded$)$, i.e. $\widetilde{\nu}_\Omega:=\pi^*\nu_\Omega=d\zeta$, where $\zeta$ is a smooth real $L^\infty_{\widetilde{\omega}}$-bounded $(2n-1)$-form on $\widetilde{X}$. Such a form $\Omega$ will be called a \textbf{Calabi--Yau hyperbolic form} $($or a \textbf{CY-hyperbolic form} for short$)$.
\end{dfn}
\begin{rmk}
    \begin{enumerate}
        \item The notion of CY-hyperbolicity does not depend on the choice of Hermitian metric $\omega$ on $X$, since $X$ is compact hence any two Hermitian metrics on $X$ are uniformly equivalent. 
        \item If one Calabi--Yau form $\Omega$ on $X$ is hyperbolic, then every Calabi--Yau form on $X$ is hyperbolic (see the discussion following Definition~\ref{7}).
        \item More generally, as we shall explain in the sequel, CY-hyperbolicity is a topological property of compact Calabi--Yau manifolds. 
    \end{enumerate}
\end{rmk}
\begin{lem}\cite[Lemma~2.4]{chen2018compact} Let $(M^m,g)$ be a compact Riemannian manifold, and let $\alpha$ be a closed $\widetilde{d}($bounded$)$ $k$-form on $M$ with $1\leq k\leq m$. Then any $\beta\in\{\alpha\}_{DR}\in H^k_{DR}(X,\R)$ is also $\widetilde{d}($bounded$)$.
\end{lem}
\begin{proof}
$\alpha$ being $\widetilde{d}($bounded$)$ means there exists a $L_{\widetilde{g}}^\infty$-bounded $(k-1)$-form $\zeta$ on $\widetilde{M}$ such that $\widetilde{\alpha}:=\pi^*\alpha=d\zeta$. Since $\beta\in\{\alpha\}_{DR}$, there exists a $(k-1)$-form $\gamma$ on $M$ such that $\beta=\alpha+d\gamma$. Then: $\widetilde{\beta}:=\pi^*\beta=\pi^*\alpha+\pi^*(d\gamma)=d(\zeta+\pi^*\gamma)$, with:
\begin{equation*}
\lVert\zeta+\pi^*\gamma\rVert_{L_{\widetilde{g}}^\infty(\widetilde{M})}\leq\lVert\zeta\rVert_{L_{\widetilde{g}}^\infty(\widetilde{M})}+\lVert\gamma\rVert_{L_g^\infty(M)}<+\infty.
\end{equation*}
\end{proof}
\begin{dfn}\cite{brunnbauer2024atoroidal} Let $(M^m,g)$ be a compact Riemannian manifold. A cohomology class $\{\alpha\}\in H^k_{DR}(M,\R)$ is \textbf{hyperbolic} if $\alpha$ is $\widetilde{d}($bounded$)$. We denote the space of \textbf{hyperbolic cohomology classes} of degree $k$ by $V^k_{\mathrm{hyp}}(M)\subset H^k_{DR}(M,\R)$.
\end{dfn}
It is straightforward to check that $V^k_{\mathrm{hyp}}(M)$ is a vector subspace of $H^k_{DR}(M,\R)$.

\hspace{2pt} Note that when $(M^m,g)$ is a closed connected oriented Riemannian manifold, then either $V^m_{\mathrm{hyp}}(M)=0$ or $V^m_{\mathrm{hyp}}(M)=H^m_{DR}(M,\R)$. The second case is precisely the case where $\widetilde{M}$ is open at infinity, which is equivalent to the positivity of the Cheeger constant $h_{\widetilde{g}}(\widetilde{M})$. By Remark~\ref{2}, the positivity of $h_{\widetilde{g}}(\widetilde{M})$ is independent of the choice of Riemannian metric $g$ on $M$. Thus, by Theorem~\ref{8}, this occurs if and only if $\pi_1(M)$ is non-amenable. We therefore obtain the following:
\begin{prop}
    Let $X$ be a compact Calabi--Yau manifold. Then $X$ is CY-hyperbolic if and only if $\pi_1(X)$ is non-amenable. \label{9}
\end{prop}
\hspace{2pt} Another immediate consequence of the definition, is the vanishing of the \textbf{\textit{L}$^2$-Betti numbers} at degree $0$ and degree $2n$. The space of square-integrable differential $k$-forms on $\widetilde{X}$ is denoted by $L_k^2(\widetilde{X})$. The \textbf{reduced \textit{L}$^2$-cohomology} of $\widetilde{X}$ is defined by:
\[
\overline{H}_{(2)}^k(\widetilde{X}):=\operatorname{Ker}(d_k)\big/\overline{\operatorname{Im}(d_{k-1})}\hspace{2pt},\quad\text{ for }0\leq k\leq 2n,
\] 
where the differential $d_k:L^2_k(\widetilde{X})\longrightarrow L^2_{k+1}(\widetilde{X)}$ is understood in the $L^2$ sense and the closure is taken in the $L^2$-norm. The deck transformation group $\pi_1(X)$ acts on $\overline{H}_{(2)}^k(\widetilde X)$, and the $k$-th $L^2$-Betti number is defined by $b_k^{(2)}(\widetilde X):=\dim_{\pi_1(X)}\overline{H}_{(2)}^k(\widetilde X)$, where $\dim_{\pi_1(X)}$ denotes the \textbf{von Neumann dimension}. These invariants are independent of the choice of Riemannian metric on the compact quotient (see~\cite[Chapter 1]{luck2002l2} for a general reference). 

\hspace{2pt} By the $L^2$-Hodge-de Rham Theorem~\cite[Theorem 1.59]{luck2002l2}, $\overline{H}_{(2)}^k(\widetilde{X})$ is naturally isomorphic to the space of \textbf{\textit{L}$^2$-harmonic \textit{k}-forms} $\mathcal{H}_{(2)}^k(\widetilde{X}):=\operatorname{Ker}\widetilde{\Delta}_k\subset L^2_k(\widetilde{X})$, where:
\[
\widetilde{\Delta}_k=dd^*+d^*d
\] 
is the self-adjoint $L^2$-Hodge Laplacian acting on $L_k^2$-forms in $\widetilde{X}$. In particular:
\[
b_k^{(2)}(\widetilde{X})=0 \Longleftrightarrow\operatorname{Ker}\widetilde{\Delta}_k=\{0\}\hspace{2pt},\quad\text{ for any }0\leq k\leq 2n.
\]
Cheeger's inequality~\cite{cheeger1970lower} yields the following lower bound for the spectrum of $\widetilde{\Delta}_0$: 
\[
\lambda\geq \frac{1}{4}h(\widetilde{X})^2.
\]
Hence, the $L^2$-Betti numbers vanish in degrees $0$ and $2n$ (the latter follows from Poincaré duality for $L^2$-Betti numbers~\cite[Theorem~1.35]{luck2002l2}).
\begin{cor}
    If $X$ is a compact CY-hyperbolic manifold, then: $b_0^{(2)}=b_{2n}^{(2)}=0$.
\end{cor}
\hspace{2pt} In view of Proposition~\ref{9} and Corollary~\ref{6} of the Beauville--Bogomolov decomposition~\ref{5}, CY-hyperbolicity is a genuinely non-Kähler hyperbolicity notion. Meaning that:
\begin{cor}
No compact Kähler Calabi--Yau manifold is CY-hyperbolic. \label{16}
\end{cor}
\hspace{2pt} We are led to the following natural question: can a compact CY-hyperbolic manifold be in \textit{Fujiki class $\mathcal{C}$}? Recall that a compact complex manifold $X$ belongs to \textbf{Fujiki class $\mathcal{C}$} if it is bimeromorphic to a compact Kähler manifold. Equivalently, there exists a \textbf{modification} of compact complex manifolds $\mu:Y\longrightarrow X$ where $Y$ is Kähler. i.e. $\mu$ is a proper holomorphic map restricting to a biholomorphism outside a proper analytic subset of $X$. In general, modifications do not preserve the triviality of the canonical bundle. For instance, if $\mu$ is the blow-up of a smooth submanifold of codimension $d\geq 2$, with exceptional divisor $E$, then: $K_Y\simeq \mu^*K_X\otimes\mathcal{O}_Y((d-1)E)$. Thus, even if $K_X$ is trivial, $K_Y$ need not be trivial, and vice versa. However, we have:
\begin{enumerate}
    \item[\textbf{(\textit{a})}] The \textit{Kodaira dimension} is a bimeromorphic invariant: Let $X$ be a compact complex manifold with $\dim_\C X=n$.  For each integer $m\geq 1$, the \textbf{\textit{m}-th plurigenus} of $X$ is defined by
    \[
    P_m(X):=h^0(X,K_X^{\otimes m})=\dim_{\C}H^0(X,K_X^{\otimes m}).
    \]
    The \textbf{Kodaira dimension} of $X$, denoted by $\kappa(X)$, is defined as follows. If $P_m(X)=0$ for every $m\geq 1$, then $\kappa(X)=-\infty$. Otherwise, for every $m\geq 1$ such that $P_m(X)>0$, consider the \textit{m-th pluricanonical map}
    \[
    \Phi_m:X\longdashrightarrow\CP\bigl(H^0(X,K_X^{\otimes m})^*\bigr).
    \]
    Then: $\kappa(X):=\underset{{\substack{m\geq 1\\ P_m(X)>0}}}\max\dim_\C \Phi_m(X)$. In particular, $\kappa(X)\in\{-\infty,0,1,\ldots,n\}$. 
    
    If $K_X\simeq\mathcal O_X$, then $K_X^{\otimes m}\simeq\mathcal O_X$ for every $m\geq 1$. Since $X$ is compact and connected, we have $h^0(X,\mathcal O_X)=1$, and hence $P_m(X)=1$ for every $m\geq 1$. Therefore, $\kappa(X)=0$.

    If $X$ and $Y$ are two bimeromorphic compact complex manifolds, then one has~\cite[Corollary 6.4]{ueno1975classification}:
    \[
    P_m(X)=P_m(Y)\hspace{2pt},\quad\text{ for every }m\geq 1\quad\text{ and }\quad\kappa(X)=\kappa(Y).
    \]
    \item[\textbf{(\textit{b})}] Fundamental groups are invariant under modifications: 
    \begin{lem}
    Let $X$ and $Y$ be compact complex manifolds, and let $\mu:Y \longrightarrow X$ be a modification. Then the induced homomorphism $\mu_*:\pi_1(Y) \longrightarrow \pi_1(X)$ between the fundamental groups is an isomorphism. \label{10}
    \end{lem}
    We include a proof of this standard fact for the sake of completeness.
    \begin{proof}
    When $X$ is a Riemann surface, then $Y$ is also a Riemann surface, and $\mu:Y\longrightarrow X$ is a proper bimeromorphic holomorphic map between compact Riemann surfaces. Since compact Riemann surfaces are normal, every bimeromorphic map between them is biholomorphic. Hence $\mu$ is a biholomorphism, and therefore $\mu_*:\pi_1(Y)\longrightarrow\pi_1(X)$ is an isomorphism. We may thus assume that $\dim_\C X\geq2$.
    
    Let $E\subset X$ be an analytic subset with $\operatorname{codim}_\C E\geq 2$ such that $\mu$ restricts to a biholomorphism $\mu|_{Y\setminus E'}:Y\setminus E'\longrightarrow X\setminus E$, where $E':=\mu^{-1}(E)$ with $\operatorname{codim}_\C E'\geq 1$. Since $E$ and $E'$ are analytic subsets, they admit stratifications by smooth locally closed submanifolds. Every stratum $S_\alpha$ of $E$ has $\operatorname{codim_\R S_\alpha\geq 4}$ in $X$, while every stratum $S'_\beta$ of $E'$ has $\operatorname{codim}_\R S'_\beta\geq 2$ in $Y$. By applying~\cite[Théorème~2.3, Chapitre IX]{Godbillon1971elements} to these strata, the inclusion $j_X:X\setminus E\xhookrightarrow{\quad} X$ induces an isomorphism on fundamental groups, whereas the inclusion $j_Y:Y\setminus E'\xhookrightarrow{\quad} Y$ induces a surjection on fundamental groups. These maps fit into the commutative diagram:
     \[
    \begin{tikzcd}[column sep=large, row sep=large]
    \pi_1(Y\setminus E')
    \arrow[r, "{(\mu|_{Y\setminus E'})_*}", "\cong"']
    \arrow[d, "{(j_Y)_*}"']
    &
    \pi_1(X\setminus E)
    \arrow[d, "{(j_X)_*}", "\cong"'] \\
    \pi_1(Y)
    \arrow[r, "\mu_*"']
    &
    \pi_1(X)
    \end{tikzcd}
    \]
    The composition $(j_X)_*\circ(\mu_{|Y\setminus E'})_*$ is an isomorphism. Since $\mu\circ(j_Y)_*=(j_X)_*\circ(\mu_{|Y\setminus E'})_*$ and $(j_Y)_*$ is surjective, it follows that $\mu_*$ is surjective.
    
    To see that $\mu_*$ is injective, let $[\gamma]\in\operatorname{Ker}\mu_*$. Since $(j_Y)_*$ is surjective, there exists $[\widetilde\gamma]\in\pi_1(Y\setminus E')$ such that $(j_Y)[\widetilde\gamma]=[\gamma]$. By the commutativity of the diagram, we then have $(j_X)_*\circ(\mu_{|Y\setminus E'})_*[\widetilde\gamma]=\mu_*\circ(j_Y)_*[\widetilde\gamma]=\mu_*[\gamma]=1$. Since $(j_X)_*$ is an isomorphism, it follows that $(\mu_{|Y\setminus E'})_*[\widetilde\gamma]=1$. As $\mu_{|Y\setminus E'}$ is a biholomorphism, it induces an isomorphism on fundamental groups, and hence $[\widetilde\gamma]=1$. Therefore $[\gamma]=(j_Y)_*[\widetilde\gamma]=1$, so $\operatorname{Ker}\mu_*={1}$. Thus $\mu_*:\pi_1(Y)\longrightarrow\pi_1(X)$ is an isomorphism.
\end{proof}
\end{enumerate}
\hspace{2pt} In complex dimensions at most three, the virtual abelianness of the fundamental group of a compact Calabi--Yau manifold in Fujiki class~$\mathcal{C}$ follows by combining~\cite[Theorem~1.3 and Proposition~3.9]{campana2014abelianity} with~\cite[Corollary~3.32]{campana2004orbifolds}. A recent preprint \cite{OpenAIAbelianitySpecialKahler2026}, which claims a proof of Campana's abelianity conjecture~\cite[Conjecture~7.1]{campana2004orbifolds} in arbitrary dimensions, allows us to extend this conclusion to all dimensions.
\begin{thm}
Let $X$ be a compact Calabi--Yau manifold in Fujiki class~$\mathcal{C}$. Then $\pi_1(X)$ is virtually abelian, and hence amenable. \label{thm:CY-amenability}
\end{thm}
\begin{proof}
Since $X$ belongs to Fujiki class~$\mathcal{C}$, there exists a modification $\mu:Y\longrightarrow X$ from a compact Kähler manifold $Y$. As modifications preserve the Kodaira dimension, we have $\kappa(Y)=\kappa(X)=0$. By~\cite[Theorem~5.1]{campana2004orbifolds}, $Y$ is therefore special. Applying the result claimed in~\cite{OpenAIAbelianitySpecialKahler2026} to $Y$, we conclude that $\pi_1(Y)$ is virtually abelian. Since modifications induce isomorphisms on fundamental groups, we have $\pi_1(X)\simeq\pi_1(Y)$, so $\pi_1(X)$ is also virtually abelian and hence amenable. In particular, no compact Calabi--Yau manifold in Fujiki class~$\mathcal{C}$ is Calabi--Yau hyperbolic by Proposition~\ref{9}.
\end{proof}
\hspace{2pt} It is worth noting that, although Corollary~\ref{16} and Theorem~\ref{thm:CY-amenability} impose restrictions on compact Calabi--Yau hyperbolic manifolds, they do not preclude their fundamental groups from being Kähler groups: both possibilities occur among the examples discussed above. On the one hand, a cocompact torsion-free lattice $\Gamma\subset\mathrm{SL}_2(\mathbb{C})$ is not a Kähler group, since it is isomorphic to the fundamental group of a closed real hyperbolic $3$-manifold, which cannot be a Kähler group by a theorem of J. Carlson and D. Toledo \cite{carlson1989harmonic}. On the other hand, Calabi--Yau hyperbolic manifolds arising as generalized Calabi--Gray manifolds can have Kähler fundamental groups. For instance, the manifold resulting from the classical Calabi--Gray construction is diffeomorphic to the product of a compact Riemann surface of genus $g\geq 3$ and a compact hyperkähler manifold, hence its fundamental group is Kähler. 

\hspace{2pt} We now turn to non-Kähler Calabi--Yau manifolds and examine their hyperbolicity in the sense of CY-hyperbolicity introduced in this paper.
\begin{ex}
    \begin{enumerate}
        \item The only compact Calabi--Yau surfaces are complex tori, K3 surfaces, and primary Kodaira surfaces. The first two are Kähler and hence are not CY-hyperbolic by the preceding corollary. A primary Kodaira surface has a fundamental group $\Gamma$ admitting the following presentation (see~\cite{kodaira1964structure}, \S6):
        
        \[
        \Gamma=\langle\delta_1,\delta_2,\delta_3,\gamma\hspace{1pt}\lvert\hspace{1pt}[\delta_1,\delta_2]=\delta_3^r\hspace{1pt},\hspace{2pt}[\delta_i,\delta_3]=[\delta_i,\gamma]=1\hspace{1pt},\hspace{2pt}\text{for }i=1,2,3\rangle\hspace{2pt},\quad r\in\N_{>0}.
        \]
        
        The group $\Gamma$ is two-step nilpotent and finitely generated. It is a classical result that every finitely generated nilpotent group has polynomial growth~\cite{milnor1968growth,wolf1968growth}, hence amenable by Remark~\ref{3}. By our characterization in Proposition~\ref{9} of CY-hyperbolicity, no compact Calabi--Yau surface is CY-hyperbolic. Consequently, complex dimension three is the lowest possible dimension for a CY-hyperbolic manifold.
        \item The Friedman--Lu--Tian complex structures on $\mathop{\#}_k(\Sp^3\times\Sp^3)$, for $k\geq 2$, and the Fine--Panov complex structures on $2(\Sp^3\times\Sp^3)\mathop{\#}(\Sp^2\times\Sp^4)$ yield Calabi--Yau manifolds. However, since these manifolds are simply connected, they are not CY-hyperbolic.
        \item Let $X=G/\Gamma$ be a compact complex parallelizable manifold, where $G$ is a simply connected complex Lie group and $\Gamma\subset G$ is a cocompact lattice. The CY-hyperbolicity of $X$ is therefore determined by the amenability of $\Gamma$. For instance, consider the Iwasawa manifold $X=H_\C/\Gamma$. Since $H_{\C}$ is nilpotent, so is $\Gamma$, and hence $\Gamma$ is amenable. Consequently, the Iwasawa manifold is not CY-hyperbolic.
        
        In contrast, let $\Gamma\subset \mathrm{SL}_2(\C)$ be a torsion-free cocompact lattice. Since $\mathrm{SL}_2(\C)$ is a non-compact semisimple Lie group, the Borel density theorem implies that $\Gamma$ is Zariski dense. By the Tits alternative, $\Gamma$ contains a non-abelian free subgroup and is therefore non-amenable (see~\cite{borel1960density,tits1972free}). Consequently, the compact complex parallelizable threefold $\mathrm{SL}_2(\C)/\Gamma$ is CY-hyperbolic.
        
        More generally, every \textbf{Yachou manifold} $X=G/\Gamma$, where $G$ is a simply connected semisimple complex Lie group and $\Gamma\subset G$ is a cocompact lattice~\cite{yachou1998varietes}, is CY-hyperbolic, since $\Gamma\simeq\pi_1(X)$ is non-amenable by the same arguments. Moreover, Yachou manifolds are known to be \textbf{degenerate balanced}~\cite[Proposition 5.4]{popovici2015aeppli}: they admit a Hermitian metric $\omega$ such that $\omega^{n-1}$ is $d$-exact (proved in~\cite[Proposition 18]{yachou1998varietes}). Hence, they are also balanced hyperbolic in the sense of~\cite[Definition 2.1]{marouani2023balanced}.

        However, no complex Lie group is Brody hyperbolic. Indeed, if $G$ is a complex Lie group with Lie algebra $\mathfrak g$, then any nonzero element $\xi\in\mathfrak g$ defines a nonconstant entire curve $f_\xi:\C\longrightarrow G$, $z\longmapsto\exp(z\xi)$. Consequently, no compact complex-parallelizable manifold is Brody hyperbolic (or, equivalently, Kobayashi hyperbolic, by Brody's theorem~\ref{4}).
        \item The generalized Calabi--Gray manifolds constructed by T. Fei~\cite{fei2024remarks} provide examples of Calabi--Yau manifolds whose CY-hyperbolicity is determined by the fundamental group of the base. Indeed, since such a manifold $X$ is diffeomorphic to $M\times Y$, where $M$ is a compact hyperkähler manifold, we have $\pi_1(X)\simeq\pi_1(M)\times\pi_1(Y)$. As $\pi_1(M)$ is amenable, our characterization~\ref{9} of CY-hyperbolicity implies that $X$ is CY-hyperbolic if and only if $\pi_1(Y)$ is non-amenable.
        
        In the classical Calabi--Gray construction~\cite{calabi1958construction,gray1969vector}, the base $Y$ is a compact Riemann surface of genus $g$ arising as a non-flat minimal surface immersed in a flat three-torus, and $\varphi: Y\longrightarrow\CP^1$ is its Gauss map. Since every compact minimal surface of genus $g=1$ in a flat three-torus is flat, while Meeks's result~\cite[Corollary~3.1]{meeks1990theory} excludes compact minimal surfaces of genus $g=2$, the non-flatness of the surface in question forces $g\geq 3$. Consequently, $\pi_1(Y)$ is non-amenable, and the resulting manifold is CY-hyperbolic.
        
        In contrast, for higher-dimensional bases, the fundamental group may be amenable. For instance, the simply connected elliptic surface $Y=E(2m+2)$, with elliptic fibration $\varphi:Y\longrightarrow\CP^1$, satisfies $K_Y\simeq\varphi^*\mathcal{O}_{\CP^1}(2m)$ by the canonical bundle formula for elliptic surfaces~\cite{kodaira1964structure}. Hence it meets Fei's canonical-bundle condition, but $\pi_1(Y)=0$, so the associated generalized Calabi--Gray manifold is not CY-hyperbolic.
        \item \textbf{\textit{p}-contact metric hyperbolicity.} A compact complex manifold $X$ of dimension $n=2p+1$, endowed with a holomorphic $p$-contact structure $\Gamma$, is called \textbf{\textit{p}-contact metrically hyperbolic}~\cite[Definition 6.1]{kasuya2025properties} if there exist a Hermitian metric $\omega$ and a real $d$-closed, $\widetilde{d}$(bounded) $(p,p)$-form $\Omega$ such that:
        \[
        \frac{\omega^p}{p!}=i^{p^2}\Gamma\wedge\overline{\Gamma}+\Omega.
        \]
        \hspace{2pt} \hspace{2pt} Compact quotients of $\mathrm{SL}_2(\C)$ provide examples of Calabi--Yau hyperbolic manifolds carrying holomorphic contact structures that are not $1$-contact metrically hyperbolic. Indeed, let $X=\mathrm{SL}_2(\C)/\Gamma$ be a compact quotient by a cocompact torsion-free lattice $\Gamma\subset\mathrm{SL}_2(\C)$. Then $X$ carries a holomorphic contact structure defined by a holomorphic $1$-form $\alpha$. Let $\pi:G\longrightarrow X$ be the quotient map, where $G=\mathrm{SL}_2(\C)$. Choose a nonzero vector $v\in\operatorname{Ker}(\pi^*\alpha)_e\subset T_eG$, where $e$ denotes the identity element of $G$. Let $V$ be the right-invariant holomorphic vector field on $G$ determined by $v$. Since $V$ is invariant under the right action of $\Gamma$, it descends to a nowhere-vanishing holomorphic vector field on $X$, which we also denote by $V$. Consider the holomorphic function $h:=\alpha(V):X\longrightarrow\C$. Since $X$ is compact and connected, $h$ is constant. Moreover, by our choice of $v$, we have $h(\pi(e))=(\pi^*\alpha)_e(v)=0$, and hence $\alpha(V)\equiv0$ on $X$. Fixing $g_0\in G$, we can therefore define a nonconstant holomorphic immersion $f:\C\longrightarrow X$, $f(z)=\pi\bigl(\exp(zv)g_0\bigr)$, which is \textbf{horizontal} since $f'(z)=V_{f(z)}$, i.e. $f^*\alpha=0$. \cite[Theorem~6.4]{kasuya2025properties} implies that $X$ is not $1$-contact metrically hyperbolic. It remains an open question whether the converse holds:
        \begin{qst}
            Does $p$-contact metric hyperbolicity imply Calabi--Yau hyperbolicity?
        \end{qst}
    \end{enumerate}
\end{ex}
\hspace{2pt} Calabi--Yau hyperbolic manifolds arising from T. Fei's construction of generalized Calabi--Gray manifolds are likewise not Kobayashi hyperbolic.
\begin{thm}
Let $X=Y\times_{\CP^1}\mathcal{Z}_M$ be a generalized Calabi--Gray manifold. Then $X$ is Kobayashi non-hyperbolic. \label{15}
\end{thm}
\begin{proof}
Let $p_Y:X\longrightarrow Y$ denote the natural projection onto the first factor. For any $y\in Y$, the fiber of $p_Y$ over $y$ is biholomorphic to the fiber $\pi^{-1}(\varphi(y))$ of the twistor projection. By the construction of the twistor space, this fiber is the hyperkähler manifold $M$ equipped with the complex structure parametrized by $\varphi(y)$.

Since $\varphi$ is nonconstant and holomorphic, the open mapping theorem implies that $\varphi(Y)$ is open in $\CP^1$. Moreover, as $Y$ is compact, $\varphi(Y)$ is compact and hence closed in $\CP^1$. Since $\CP^1$ is connected, it follows that $\varphi(Y)=\CP^1$.

By F. Campana's theorem~\cite[Theorem 1]{Campana1992twistor}, at least one fiber of the twistor projection $\pi$ contains a nonconstant entire holomorphic curve. Thus, there exists $t_0\in\CP^1$ and a nonconstant holomorphic map $g:\C\longrightarrow\pi^{-1}(t_0)$. By the surjectivity of $\varphi$, we can choose $y_0\in Y$ such that $\varphi(y_0)=t_0$. The fiber $p_Y^{-1}(y_0)$ is then biholomorphic to $\pi^{-1}(t_0)$, so $g$ composes with the natural inclusion $p_Y^{-1}(y_0)\xhookrightarrow{\quad} X$ of the fiber into $X$, inducing a nonconstant entire curve in $X$. Consequently, $X$ is Kobayashi non-hyperbolic by Brody's criterion~\ref{4}.
\end{proof}
\begin{rmk}
Assuming conjecture~\ref{1}, claimed to be proved in the preprint~\cite{OAI-Kobayashi-2026} but not yet independently validated, the theorem admits a shorter alternative proof. If $X$ were Kobayashi hyperbolic, then each fiber $M_y:=p_Y^{-1}(y)$ would also be Kobayashi hyperbolic. But $M_y$ is a compact Kähler manifold with trivial canonical bundle, yielding a contradiction.
\end{rmk}
\hspace{2pt} The examples discussed above suggest that CY-hyperbolicity and Kobayashi hyperbolicity may be incompatible. This leads us to the following natural question.
\begin{qst}
Does there exist a compact Calabi--Yau hyperbolic manifold that is also  Kobayashi hyperbolic? \label{14}
\end{qst}
\hspace{2pt} Assuming the validity of the proof of the GSS conjecture in~\cite{OAI-GSS-2026}, a stronger obstruction to Kobayashi hyperbolicity holds in complex dimension two:
\begin{thm}
Assuming the \textbf{GSS conjecture}, every compact Kobayashi hyperbolic complex surface is Kähler.
\end{thm}
\begin{proof}
We use the Enriques--Kodaira classification, according to which the non-Kähler compact complex surfaces are either non-Kähler elliptic surfaces, Kodaira surfaces, or surfaces of class~VII (see, for instance,~\cite{barth2004compact}). Non-Kähler elliptic surfaces admit elliptic fibrations and therefore contain smooth elliptic curves. Kodaira surfaces likewise admit elliptic fibrations. Since every elliptic curve is covered by $\C$, these surfaces cannot be Kobayashi hyperbolic. 

It remains to consider surfaces of class~VII. These are compact complex surfaces with first Betti number equal to one. If the second Betti number vanishes, the classification theorem of Bogomolov implies that a minimal surface of class~VII is either a \textit{Hopf surface} or an \textit{Inoue surface}~\cite{bogomolov1976classification,inoue1974surfaces}. Neither type is Kobayashi hyperbolic. Indeed, the universal cover of a Hopf surface is biholomorphic to $\C^2\setminus\{0\}$, which contains, for instance, the nonconstant entire curve $z\longmapsto (e^z,0)$. Similarly, the universal cover of an Inoue surface is biholomorphic to $\Hyp\times\C$, where $\Hyp$ denotes the upper half-plane. Fixing $w_0\in\Hyp$, the map $z\longmapsto (w_0,z)$ defines a nonconstant entire curve in the universal cover. Composing these maps with the corresponding covering projections yields nonconstant entire curves in the surfaces themselves.

The remaining case consists of minimal surfaces of class~VII with positive second Betti number. The \textbf{GSS conjecture} asserts that every such surface is a \textit{Kato surface}, i.e. a compact complex surface with positive first Betti number that has a \textit{global spherical shell} (see~\cite{dloussky2003class}), i.e. a holomorphically embedded neighborhood of $\Sp^3$ in $\C^2\setminus\{0\}$ whose complement is connected. In particular, a minimal Kato surface contains exactly $b_2$ irreducible rational curves. Since rational curves preclude Kobayashi hyperbolicity, these surfaces are Kobayashi non-hyperbolic either. Finally, nonminimal surfaces contain exceptional rational curves and are therefore Kobayashi non-hyperbolic.   
\end{proof}
\hspace{2pt} In higher dimensions, the question remains open. Assuming the validity of the proof of Kobayashi's conjecture~\ref{1} given in~\cite{OAI-Kobayashi-2026}, the weaker form of S. Kobayashi's conjecture for Calabi--Yau manifolds follows: no compact Kähler Calabi--Yau manifold is Kobayashi hyperbolic. This conclusion should be distinguished from the original conjecture of S. Kobayashi stating that the Kobayashi pseudodistance vanishes identically on every compact Kähler Calabi--Yau manifold.

\hspace{2pt} Since CY-hyperbolic manifolds are non-Kähler, the Kähler result above does not settle our original Question~\ref{14}.

\hspace{2pt} In parallel with Question~\ref{14}, we aim to highlight the relation between CY-hyperbolicity and a corresponding geometric notion based on subexponential growth introduced in \cite[Definition~2.3]{marouani2023balanced} and subsequently extended to higher dimensions in \cite{haggui2023compact,kasuya2025partially,khelifati2026deformations}. In the spirit of the classical implication from Kähler hyperbolicity to Kobayashi hyperbolicity, we investigate whether CY-hyperbolicity rules out non-degenerate holomorphic maps from $\mathbb{C}^n$ into $X$ satisfying this subexponential growth condition.

\hspace{2pt} Let $(X,\omega)$ be a compact Hermitian manifold with $\dim_\C X=n\geq 2$, and let $f:\C^n\longrightarrow X$ be a holomorphic map which is \textit{non-degenerate} at some point $z_0\in\C^n$, i.e. $d_{z_0}f:\C^n\longrightarrow T^{1,0}_{z_0}X$ is of maximal rank. Set: $\Sigma_f:=\{z\in\C^n\hspace{1pt}:\hspace{1pt}\operatorname{rank}\left(d_{z}f\right)\leq n-1\}$; this is a proper analytic subset of $\C^n$. Then $f^*\omega$ defines a \textit{degenerate metric} on $\C^n$ with degeneration set $\Sigma_f$ ; that is $f^*\omega\geq 0$ on $\C^n$ and defines a genuine Hermitian metric on $\C^n\setminus\Sigma_f$. For any $r>0$, we define the ($\omega,f$)\textbf{-volume} of the ball $\B_r\subset\C^n$ of radius $r$ centered at $0$ to be:
\begin{equation}
    \operatorname{Vol}_{\omega,f}(\B_r):=\int_{\B_r}f^*\omega_n>0\hspace{2pt},\quad\qquad\text{where }\hspace{2pt}\omega_n=\dfrac{\omega^n}{n!}
\end{equation}
For any $z\in\C^n$, let $\beta(z)=\lvert z\rvert^2$ be its squared Euclidean norm. By definition of the Hodge star operator induced by the Hermitian metric $f^*\omega$ on $\C^n\setminus\Sigma_f$, we have: 
\begin{equation}
    \dfrac{d\beta(z)}{\lvert d\beta(z)\rvert_{f^*\omega}}\wedge*_{f^*\omega}\left(\dfrac{d\beta(z)}{\lvert d\beta(z)\rvert_{f^*\omega}}\right)=f^*\omega_n(z)\hspace{2pt},\qquad\text{for any } z\in\C^n\setminus\Sigma_f
\end{equation}
This implies that the ($2n-1$)-form $d\sigma_{\omega,f}:=*_{f^*\omega}\left(\dfrac{d\beta(z)}{\lvert d\beta(z)\rvert_{f^*\omega}}\right)$ defines an area measure on spheres $\Sp_r:=\partial\B_r\subset\C^n$. In particular, we define the ($\omega,f$)\textbf{-area} of the sphere $\Sp_r$ of radius $r>0$ by:
\begin{equation}
    A_{\omega,f}(\Sp_r):=\int_{\Sp_r}d\sigma_{\omega,f,r}>0\hspace{2pt},\qquad\text{ where }\hspace{2pt}d\sigma_{\omega,f,r}=\left(d\sigma_{\omega,f}\right)_{\lvert\Sp_r}
\end{equation}
\begin{dfn}[Definition 2.3, \cite{marouani2023balanced} and Definition 2.3, \cite{kasuya2025partially}]
In the same setting, we say that $f$ has \textbf{subexponential growth} if:
    \renewcommand{\labelenumi}{(\textit{\roman{enumi}})}
    \begin{enumerate}
    \item There exist $C_0$ and $r_0>0$ such that: $\int_{\Sp_r}\lvert d\beta\rvert_{f^*\omega}d\sigma_{f,\omega,r}\leq C_0 r \operatorname{Vol}_{\omega,f}(\B_r)$, $\forall\hspace{1pt}r>r_0$.
    \item For any $\lambda>0$, we have: $\underset{t\rightarrow+\infty}{\limsup}\left(\dfrac{t}{\lambda}-\log\left(\int_0^t\operatorname{Vol}_{\omega,f}(\B_r)dr\right)\right)=+\infty$.
    \end{enumerate} \label{20}
\end{dfn}
\begin{dfn}\cite[$p=n$]{haggui2023compact,kasuya2025partially,khelifati2026deformations}
We say that $X$ is $n$\textbf{-cyclically hyperbolic} if there is no holomorphic map $f:\C^n\longrightarrow X$ that is non-degenerate at some point $z_0\in\C^n$ and has subexponential growth.
\end{dfn}
We now prove the following:
\begin{thm}
Any compact complex manifold $X$ with non-amenable fundamental group $\pi_1(X)$ is $n$-cyclically hyperbolic, where $n=\dim_\C X$. In particular, any Calabi--Yau manifold $X$ with $\dim_\C X=n$ is $n$-cyclically hyperbolic.
\end{thm}
\begin{proof}
Let $(X,\omega)$ be a compact Hermitian manifold with $\dim_C X=n$, and let $\pi:\widetilde{X}\longrightarrow X$ be its universal covering map. Set $\widetilde{\omega}:=\pi^*\omega$ and $\omega_n:=\omega^n/n!$. Since $\pi_1(X)$ is non-amenable, the lifted volume form $\widetilde{\omega}_n:=\pi^*\omega_n$ admits a smooth real bounded primitive. In other words, there exists a smooth real $L_{\widetilde{\omega}}^\infty$-bounded $(2n-1)$-form $\zeta$ on $\widetilde{X}$ such that $\widetilde{\omega}_n=d\zeta$.

Assume, by contradiction, that there exists a holomorphic map $f:\C^n\longrightarrow X$ which is non-degenerate at some point and has subexponential growth in the sense of Definition~\ref{20}. Since $\C^n$ is simply connected, $f$ lifts to a holomorphic map $\widetilde{f}:\C^n\longrightarrow\widetilde{X}$ satisfying $\pi\circ\widetilde{f}=f$. The degeneracy locus of $\widetilde{f}$ is precisely $\Sigma_f$. Hence, $f^*\omega$ defines a genuine Hermitian metric on $\C^n\setminus\Sigma_f$. We first make the following observation:

\begin{cl}
Let $\zeta$ be a smooth $\widetilde{\omega}$-bounded $k$-form on $\widetilde{X}$, where $k\leq 2n$. Then $\widetilde{f}^*\zeta$ is bounded with respect to the degenerate metric $f^*\omega$ on $\C^n\setminus\Sigma_f$.
\label{21}
\end{cl}
\begin{proof}
Let $z\in\C^n\setminus\Sigma_f$ and let $\xi_1,\ldots,\xi_k\in T_z\C^n$. Then
\begin{equation*}
\left|\widetilde{f}^*\zeta(\xi_1,\ldots,\xi_k)\right|^2_{f^*\omega}
=
\left|\zeta\bigl(\widetilde{f}_*\xi_1,\ldots,\widetilde{f}_*\xi_k\bigr)\right|^2_{\widetilde{\omega}}
\overset{(a)}{\leq}
C\prod_{j=1}^k\left|\widetilde{f}_*\xi_j\right|^2_{\widetilde{\omega}}
\overset{(b)}{=}
C\prod_{j=1}^k|\xi_j|^2_{f^*\omega},
\end{equation*}
where $(a)$ follows from the boundedness of $\zeta$, while $(b)$ follows from $\widetilde{f}^*\widetilde{\omega}=f^*\omega$.
\end{proof}

Since $\widetilde{\omega}_n=d\zeta$, we obtain:
\begin{equation}
f^*\omega_n=\widetilde{f}^*\widetilde{\omega}_n=\widetilde{f}^*(d\zeta)=d(\widetilde{f}^*\zeta).
\label{22}
\end{equation}
Combining (\ref{22}) with Stokes' theorem yields:
\begin{equation}
\operatorname{Vol}_{\omega,f}(\B_r)=\int_{\B_r}f^*\omega_n=\int_{\B_r}d(\widetilde{f}^*\zeta)=\int_{\Sp_r}\widetilde{f}^*\zeta\leq B\,A_{\omega,f}(\Sp_r), \label{23}
\end{equation}
where the last inequality follows from Claim~\ref{21}, since $\zeta$ is $\widetilde{\omega}$-bounded. Here $B>0$ is a constant independent of $r$. On the other hand, the Cauchy--Schwarz inequality yields:
\begin{equation}
A_{\omega,f}(\Sp_r)^2\leq
\left(\int_{\Sp_r}\frac{1}{|d\beta|_{f^*\omega}}\,d\sigma_{\omega,f,r}\right)
\left(\int_{\Sp_r}|d\beta|_{f^*\omega}\,d\sigma_{\omega,f,r}\right).
\label{24}
\end{equation}
Moreover, the coarea formula gives:
\begin{equation*}
\operatorname{Vol}_{\omega,f}(\B_r)=\int_0^r\left(\int_{\Sp_t}\frac{1}{|d\beta|_{f^*\omega}}\,d\sigma_{\omega,f,t}\right)2t\,dt,
\end{equation*}
where $\beta(z)=|z|^2$. Combining this identity with (\ref{23}) and (\ref{24}), we obtain:
\begin{align*}
\operatorname{Vol}_{\omega,f}(\B_r)&\geq\int_0^r\frac{2t\,A_{\omega,f}(\Sp_t)^2}{\int_{\Sp_t}|d\beta|_{f^*\omega}\,d\sigma_{\omega,f,t}}\,dt\\&\geq\frac{2}{B^2}\int_0^r\frac{t\,\operatorname{Vol}_{\omega,f}(\B_t)^2}{\int_{\Sp_t}|d\beta|_{f^*\omega}\,d\sigma_{\omega,f,t}}\,dt.
\end{align*}
By condition~(\textit{i}) in Definition~\ref{20}, there exist constants $C_0>0$ and $r_0>0$ such that:
\[
\int_{\Sp_t}|d\beta|_{f^*\omega}\,d\sigma_{\omega,f,t}\leq C_0t\operatorname{Vol}_{\omega,f}(\B_t)
\]
for every $t>r_0$. Therefore, for every $r>r_0$, we obtain:
\begin{align}
\operatorname{Vol}_{\omega,f}(\B_r)&\geq\frac{2}{B^2}\int_{r_0}^r\frac{t\,\operatorname{Vol}_{\omega,f}(\B_t)^2}{C_0t\operatorname{Vol}_{\omega,f}(\B_t)}\,dt\notag\\&\geq\frac{2}{C_0B^2}\int_{r_0}^r\operatorname{Vol}_{\omega,f}(\B_t)\,dt. \label{25}
\end{align}
Setting $K:=\dfrac{2}{C_0B^2}>0$ and $F(r):=\int_{r_0}^r\operatorname{Vol}_{\omega,f}(\B_t)\,dt$, inequality~(\ref{25}) becomes:
\[
F'(r)\geq KF(r),\qquad r>r_0.
\]
Since $f$ is non-degenerate at some point, we can choose $r_1>r_0$ such that $F(r_1)>0$. Hence:
\[
(\log F(r))'\geq K
\]
for every $r\geq r_1$. Integrating over $[r_1,r]$, we obtain:
\[
F(r)\geq F(r_1)\exp(K(r-r_1)).
\]
However, this contradicts the subexponential growth condition~(\textit{ii}) in Definition~\ref{20}. Indeed, choosing $\lambda>1/K$, the preceding exponential lower bound implies:
\[
\frac{r}{\lambda}-\log F(r)\leq\left(\frac{1}{\lambda}-K\right)r+Kr_1-\log F(r_1)\longrightarrow-\infty
\]
This contradiction shows that $X$ is $n$-cyclically hyperbolic. In particular, since every CY-hyperbolic manifold has non-amenable fundamental group by Proposition~\ref{9}, every CY-hyperbolic manifold is $n$-cyclically hyperbolic.
\end{proof}
\subsection{Construction of Calabi--Yau hyperbolic manifolds}
\hspace{2pt} Finally, we present some methods to construct new CY-hyperbolic manifolds.
\begin{prop}
Let $X$ be a compact CY-hyperbolic manifold and $Y$ any compact Calabi--Yau manifold. Then $X\times Y$ is CY-hyperbolic. \label{11}
\end{prop}
\begin{proof}
Since $K_X$ and $K_Y$ are trivial, the canonical bundle of $X\times Y$ is also trivial since: $K_{X\times Y}\simeq p_X^*K_X\otimes p_Y^*K_Y\simeq \mathcal{O}_{X\times Y}$, where $p_X$ and $p_Y$ denote the projections of $X\times Y$ onto $X$ and $Y$, respectively. Moreover, $\pi_1(X)$ is non-amenable by Proposition~\ref{9}. Since $\pi_1(X\times Y)\simeq \pi_1(X)\times\pi_1(Y)$, the fundamental group of $X\times Y$ is non-amenable as well. The result therefore follows from Proposition~\ref{9}.
\end{proof}
\hspace{2pt} A useful structural property of CY-hyperbolic manifolds is the following:
\begin{prop}
Let $X$ and $Y$ be compact Calabi--Yau manifolds and let $f:Y\longrightarrow X$ be a finite étale cover. Then $X$ is CY-hyperbolic if and only if $Y$ is CY-hyperbolic.
\end{prop}
\begin{proof}
Since $f$ is a finite étale cover, the induced homomorphism $f_*:\pi_1(Y)\xhookrightarrow{\quad}\pi_1(X)$ is injective and identifies $\pi_1(Y)$ with a subgroup of index $[\pi_1(X):f_*(\pi_1(Y))]=\deg(f)$ in $\pi_1(X)$. Hence $\pi_1(Y)$ is non-amenable if and only if $\pi_1(X)$ is non-amenable. The result follows from Proposition~\ref{9}.
\end{proof}
\begin{rmk}
The notion of CY-hyperbolicity is invariant under finite étale covers and extends naturally to compact complex manifolds with torsion canonical bundle, i.e. some positive power of the canonical bundle is trivial. Since every such manifold admits a finite étale cover with trivial canonical bundle, this extension does not enlarge the class of CY-hyperbolic manifolds up to finite étale covers.
\end{rmk}
\hspace{2pt} We now investigate the behavior of CY-hyperbolicity under bimeromorphic transformations. Since a modification does not preserve the triviality of the canonical bundle in general, we restrict our attention to \textit{crepant} modifications. More precisely, a modification $\mu:Y\longrightarrow X$ is called \textbf{crepant} if $K_Y\simeq\mu^*K_X$. In particular, crepant modifications preserve the triviality of the canonical bundle. This leads to the following result:
\begin{thm}
Let $\mu:Y\longrightarrow X$ be a crepant modification between compact complex manifolds. If $X$ is CY-hyperbolic, then $Y$ is CY-hyperbolic. \label{12}
\end{thm}
\begin{proof}
Since $\mu$ is crepant and $K_X\simeq\mathcal{O}_X$, we have $K_Y\simeq\mathcal{O}_Y$. Moreover, a modification induces an isomorphism of fundamental groups, so $\pi_1(Y)\simeq\pi_1(X)$ by Lemma~\ref{10}. Since $X$ is CY-hyperbolic, Proposition~\ref{9} implies that $\pi_1(X)$ is non-amenable. Hence $\pi_1(Y)$ is also non-amenable, and another application of Proposition~\ref{9} shows that $Y$ is CY-hyperbolic.
\end{proof}
\hspace{2pt} Deformation theory also offers a way to produce new families of CY-hyperbolic manifolds, as we shall demonstrate, as Proposition~\ref{9} implies that CY-hyperbolicity is constant throughout families of Calabi--Yau manifolds:
\begin{cor}
Let $\sigma:\mathcal{X}\longrightarrow B$ be a smooth proper holomorphic family of compact Calabi--Yau manifolds. Then either every fibre is CY-hyperbolic or none is.
\end{cor}
\begin{proof}
By Ehresmann's theorem~\cite{ehresmann1947espaces}, the fibres are diffeomorphic and hence have isomorphic fundamental groups. The result follows from Proposition~\ref{9}.
\end{proof}
\hspace{2pt} Several examples of non-Kähler Calabi--Yau $\partial\bar\partial$-manifolds have been constructed, notably in the setting of solvmanifolds; see, for instance, the works of A. Fino, A. Otal and L. Ugarte~\cite{fino2015six}, and A. Cattaneo and A. Tomassini~\cite{cattaneo2024lemma}. The latter examples show that one cannot expect the fundamental group of a compact Calabi--Yau $\partial\bar\partial$-manifold to be virtually abelian. Indeed, A. Cattaneo and A. Tomassini construct \textbf{generalized Nakamura manifolds} $N_{M,P,\tau}=G_M/\Gamma$, where $G_M=\C\ltimes_\rho\C^n$ is a solvable complex Lie group and $\Gamma=\Gamma'_\tau\ltimes_\rho\Gamma''_P$ is a cocompact lattice. For suitable choices of the parameters, these are non-Kähler Calabi--Yau $\partial\bar\partial$-manifolds~\cite[Theorem~1.1 and Section~\S4.6]{cattaneo2024lemma}. Their fundamental groups are isomorphic to $\Gamma$. To see how such a group can fail to be virtually abelian, one may take, for example,
\[
M=\begin{pmatrix}2&1\\1&1\end{pmatrix}\in\mathrm{SL}_2(\Z),
\]
whose eigenvalues are $(3+\sqrt5)/2$ and $(3-\sqrt5)/2$. In the construction, the action of a generator of the $\Z^2$-factor $\Gamma'_\tau$ on the lattice $\Gamma''_P\simeq\Z^4$ is induced by $M$ on the two copies of $\Z^2$. Since no positive power of $M$ is the identity, this action remains nontrivial for every positive power of that generator. Consequently, no finite-index subgroup of $\Gamma$ can be abelian: such a subgroup intersects $\Gamma''_P$ in a finite-index subgroup and contains a positive power of the acting generator, which does not commute with every element of that intersection. Thus, for admissible choices of $\tau$, the resulting Calabi--Yau $\partial\bar\partial$-manifold has a fundamental group that is not virtually abelian. Nevertheless, since $\Gamma$ is a lattice in a solvable Lie group, it is solvable and hence amenable. This example therefore rules out virtual abelianness as a general conclusion while remaining compatible with the weaker conjecture that $\pi_1(X)$ is amenable. Other examples of non-Kähler Calabi--Yau $\partial\bar\partial$-manifolds arise from constructions involving connected sums of copies of $\Sp^3\times\Sp^3$ and related manifolds~\cite{friedman2020lemma}; the examples under consideration here are simply connected and therefore also have amenable fundamental groups. For the general theory of the $\partial\bar\partial$-lemma and its relation to Bott--Chern cohomology, we refer~\cite{angella2013lemma}. A similar phenomenon occurs for compact complex-parallelizable manifolds, among which the $\partial\bar\partial$-property characterizes complex tori. Indeed, the $\partial\bar\partial$-property implies the degeneration of the Frölicher spectral sequence at the first page~\cite[{\S}~5, (5.21)]{deligne1975real}, while this degeneration characterizes complex tori among compact complex-parallelizable manifolds~\cite[Introduction]{kasuya2023frolicher}. In view of these observations, we are led to formulate the following variant of Theorem~\ref{thm:CY-amenability}, in which we replace the Fujiki class~$\mathcal{C}$ assumption with the $\partial\bar\partial$-property, while weakening the expected conclusion from virtual abelianity to amenability.
\begin{conj}
Let $X$ be a compact Calabi--Yau $\partial\bar\partial$-manifold. Then $\pi_1(X)$ is amenable. In particular, $X$ is not Calabi--Yau hyperbolic.
\end{conj}
\hspace{2pt} If this conjecture holds, Bogomolov--Tian--Todorov unobstructedness theorem for compact Calabi--Yau $\partial\bar\partial$-manifolds \cite[Theorem~1.2]{popovici2019holomorphic} does not apply to the CY-hyperbolic setting. Moreover, the triviality of the canonical bundle is not stable under small deformations of compact complex manifolds in general: I. Nakamura constructed a compact complex-parallelisable manifold admitting arbitrarily small deformations with Kodaira dimension $\kappa=-\infty$ \cite[\S3, Case~III-(3b)]{nakamura1975complex}.

\hspace{2pt} Nevertheless, we shall explore two distinct deformation-theoretic approaches to construct CY-hyperbolic manifolds.
\begin{enumerate}
    \item[\textbf{(\textit{a})}] \textbf{Deformations of the hyperkähler factor of a generalized Calabi--Gray manifold.} Let $M$ be a compact hyperkähler manifold and $X$ a generalized Calabi--Gray manifold associated with $M$ and a fixed pair $(Y,\varphi)$ as above. By Kodaira--Spencer deformation theory, the space of infinitesimal deformations of $M$ is given by $H^1(M,T_M)$, while obstructions to extending infinitesimal deformations may lie in $H^2(M,T_M)$~\cite{kodaira1958deformations}. By the Bogomolov--Tian--Todorov theorem~\cite{bogomolov1978hamiltonian,tian1987smoothness,todorov1989weil}, small deformations of $M$ are unobstructed. Moreover, every sufficiently small deformation of $M$ is hyperkähler~\cite[Proposition 9]{beauville1983varietes}. We can choose a small deformation family $\{M_t\}_{t\in\B_\varepsilon}$, where $\B_\varepsilon$ is an open ball of sufficiently small radius $\varepsilon>0$ centered at the origin in $\C^k$, with $k=h^1(M,T_M)$. Keeping $(Y,\varphi)$ fixed, we can construct the corresponding family of generalized Calabi--Gray manifolds $\{X_t\}_{t\in\B_\varepsilon}$, with $X_0=X$. The twistor spaces associated to $\{M_t\}_{t\in\B_\varepsilon}$ vary smoothly with the hyperkähler structures, and so do the induced complex structures on the family $\{X_t\}_{t\in\B_\varepsilon}$. In particular, each $X_t$ has trivial canonical bundle and satisfies $\pi_1(X_t)\simeq\pi_1(M)\times\pi_1(Y)$. It follows from Proposition~\ref{9} that every $X_t$ is CY-hyperbolic.
    \item[\textbf{(\textit{b})}] \textbf{Ghys's deformations of complex structures on compact quotients of $\mathrm{SL}_2(\C)$.} Although complex-parallelizable manifolds have trivial holomorphic tangent bundles, they need not be locally rigid. In particular, \textit{Raghunathan's rigidity theorem} asserts that if $G$ is a connected complex semisimple Lie group with no simple factor locally isomorphic to $\mathrm{SL}_2(\C)$ (for example $G=\mathrm{SL}_n(\C)$ with $n\geq 3$), and $\Gamma\subset G$ is a torsion-free discrete cocompact subgroup, then the canonical complex structure on the compact quotient $G/\Gamma$ is locally rigid: every sufficiently small deformation is biholomorphic to the original complex structure~\cite[Theorem~3 and its Corollary]{raghunathan1966vanishing}.
    
    \hspace{2pt} Let $G$ be a connected, simply connected complex semisimple Lie group, let $\Gamma\subset G$ be a torsion-free cocompact lattice, and set $X=G/\Gamma$. Since $X$ is complex-parallelizable, its holomorphic tangent bundle is trivial, and hence $H^1(X,T_X)\simeq H^1(X,\mathcal{O}_X)\otimes\mathfrak g$, where $\mathfrak g$ denotes the Lie algebra of $G$. Moreover, for such quotients, $H^1(X,\mathcal{O}_X)\simeq H^1(\Gamma,\C)$. Thus, if $b_1(\Gamma)=0$, then $H^1(X,T_X)=0$, so $X$ is infinitesimally rigid.
    
    \hspace{2pt} The case $G=\mathrm{SL}_2(\C)$ is particularly interesting, as it lies outside this rigidity regime and allows us to construct new CY-hyperbolic threefolds by deforming the complex structures of its compact quotients by suitable torsion-free cocompact lattices, as we shall see below.
    
    \hspace{2pt} Let $\Gamma\subset\mathrm{SL}_2(\C)$ be a torsion-free cocompact lattice and put $X_0=\mathrm{SL}_2(\C)/\Gamma$. It was proved by C. S. Rajan that the canonical complex structure on $X_0$ admits nontrivial deformations if and only if $b_1(\Gamma)>0$~\cite{Rajan1994deformations}. 
    
    \hspace{2pt} This criterion is closely connected to the topology of hyperbolic $3$-manifolds. Indeed, $\mathrm{PSL}_2(\C)$ is the identity component of the orientation-preserving isometry group of hyperbolic $3$-space, and $\Gamma$ determines a closed hyperbolic $3$-manifold $\Hyp^3/\overline{\Gamma}$, where $\overline{\Gamma}$ denotes its image in $\mathrm{PSL}_2(\C)$. Although $b_1(\Gamma)$ need not be positive, Thurston's virtual fibering conjecture (proven by I. Agol) implies that every closed hyperbolic $3$-manifold has a finite-sheeted cover with positive first Betti number~\cite{agol2013virtual}. Consequently, there exists a finite-index subgroup $\Gamma'\subset\Gamma$ with $b_1(\Gamma')>0$. Applying Rajan's theorem to $\Gamma'$ shows that the corresponding finite cover $X'_0=\mathrm{SL}_2(\C)/\Gamma'$ admits nontrivial deformations, even when $X_0$ itself is locally rigid.
    
    \hspace{2pt} Ghys's work provides a geometric description of these deformations. Let $G=\mathrm{SL}_2(\C)$ and consider the representation variety $\mathcal{R}(\Gamma)=\operatorname{Hom}(\Gamma,G)$, with $\rho_0$ denoting the trivial representation. For a representation $\rho$ sufficiently close to $\rho_0$, define an action of $\Gamma$ on $G$ by $\gamma*_\rho u=\rho(\gamma)^{-1}u\gamma$, where $\gamma\in\Gamma$ and $u\in G$. For $\rho$ sufficiently close to $\rho_0$, this action is free and properly discontinuous, and its quotient $X_\rho=G/\Gamma_\rho$ is a compact complex manifold diffeomorphic to $X_0$. When $\rho=\rho_0$, this recovers $X_0$, and for $\rho,\rho'\in\mathcal{R}(\Gamma)$ close to $\rho_0$, $X_{\rho}$ and $X_{\rho'}$ are biholomorphic if and only if $\rho=u\rho'u^{-1}$ for some $u\in G$. Moreover, $X_\rho$ is not complex-parallelizable for $\rho\neq\rho_0$.
    
    \hspace{2pt} E. Ghys relates the local deformation space of $X_0$ to the deformation space of holomorphic $(G\times G,G)$-structures: the latter is described locally by representations into $G\times G$, and in this construction the holonomy representations are of the form $\gamma\longmapsto(\rho(\gamma),\gamma)$. After accounting for the natural equivalences of representations, this gives a character-variety description of the local deformation space in terms of the germ of $\operatorname{Hom}(\Gamma,G)\big/\big/G$ at the class of the trivial representation~\cite{Ghys1995deformations}. T. Jamin gives a more recent treatment of the deformation theory of compact quotients of $\mathrm{SL}_2(\C)$ in the framework of the Teichmüller stack~\cite{Jamin2024teichmuller}. We now show that the manifolds obtained through Ghys's construction remain CY-hyperbolic.
    
    \begin{thm}
    $X_\rho$ is CY-hyperbolic for a representation $\rho$ sufficiently close to the trivial representation $\rho_0$. \label{13}
    \end{thm}
    \begin{proof}
    The complex Lie group $G$ admits a nowhere-vanishing holomorphic $3$-form $\Omega$ invariant under both left and right translations. Consequently, $\Omega$ is invariant under the deformed action $u\longmapsto\rho(\gamma)^{-1}u\gamma$:
    \[
    \gamma^*\Omega=(L_{\rho(\gamma)^{-1}}\circ R_\gamma)^*\Omega=\Omega\hspace{2pt},\qquad\text{for all }\gamma\in\Gamma_\rho
    \]
    Hence $\Omega$ descends to a nowhere-vanishing holomorphic $3$-form on $X_\rho$, giving $K_{X_\rho}\simeq\mathcal{O}_{X_\rho}$. Furthermore, $G$ is simply connected, and the deformed action is free and properly discontinuous. Therefore, $G\longrightarrow X_\rho$ is the universal covering map, and $\pi_1(X_\rho)\simeq\Gamma$ is non-amenable. The conclusion now follows from Proposition~\ref{9}.
    \end{proof}
\end{enumerate}
\hspace{2pt} We now describe a construction of new CY-hyperbolic manifolds from \textit{holomorphic principal elliptic bundles} over compact CY-hyperbolic manifolds. Recall that a \textbf{holomorphic principal elliptic bundle} is a holomorphic map $p:X\longrightarrow Y$ equipped with a holomorphic action of an elliptic curve $E$ that preserves the fibres and acts freely and transitively on each fibre. In particular, $p$ is locally holomorphically isomorphic to the projection $U\times E\to U$.
\begin{thm}
Let $Y$ be a compact CY-hyperbolic manifold, and let $p:X\longrightarrow Y$ be a holomorphic principal elliptic bundle. Then $X$ is CY-hyperbolic. \label{18}
\end{thm}
\begin{proof}
Since $p:X\longrightarrow Y$ is a holomorphic principal elliptic bundle, its fibres are complex tori of dimension one, and the vertical tangent bundle $T_{X/Y}:=\operatorname{Ker}(dp)$ is holomorphically trivial. Indeed, the Lie algebra of the elliptic curve induces a nowhere-vanishing holomorphic section of $T_{X/Y}$. The relative tangent sequence:
\[
0\longrightarrow T_{X/Y}\longrightarrow T_X
\overset{dp}{\longrightarrow}p^*T_Y\longrightarrow 0
\]
therefore yields, upon taking determinants, $K_X\simeq p^*K_Y\simeq\mathcal{O}_X$. Moreover, the homotopy exact sequence of the fibration $p:X\longrightarrow Y$ contains:
\[
\pi_1(X)\overset{p_*}{\longrightarrow}\pi_1(Y)
\longrightarrow \pi_0(E),
\]
where $E$ is the elliptic fibre. Hence $p_*:\pi_1(X)\longrightarrow\pi_1(Y)$ is surjective. Since a quotient of an amenable group is amenable and $\pi_1(Y)$ is non-amenable, it follows that $\pi_1(X)$ is non-amenable. The conclusion now follows from Proposition~\ref{9}.
\end{proof}
\hspace{2pt} We now show that the preceding proposition applies to holomorphically nontrivial principal elliptic bundles by establishing a sufficient condition for their existence.
\begin{prop}
Let $Y$ be a compact complex manifold such that $r=h^{0,1}(Y)>0$, and let $E=\C/\Lambda$ be an elliptic curve. Then there exists a holomorphically nontrivial principal $E$-bundle over $Y$ whose underlying topological bundle is trivial. \label{19}
\end{prop}
\begin{proof}
Let $\mathcal{E}$ denote the sheaf of holomorphic maps to $E$. The exact sequence of sheaves
\[
0\longrightarrow\underline{\Lambda}\longrightarrow\mathcal{O}_Y\longrightarrow\mathcal{E}\longrightarrow0
\]
induces an injection
\[
Q_E(Y):=
\frac{H^1(Y,\mathcal{O}_Y)}{\operatorname{Im}\bigl(H^1(Y,\underline{\Lambda})\longrightarrow H^1(Y,\mathcal{O}_Y)\bigr)}\hookrightarrow H^1(Y,\mathcal{E}).
\]
The group $H^1(Y,\underline{\Lambda})$ is countable, since $Y$ is compact and $H^1(Y,\Z)$ is finitely generated, whereas $H^1(Y,\mathcal{O}_Y)\simeq\C^r$ is uncountable by hypothesis. It follows that $Q_E(Y)\neq0$. Any nonzero class in $Q_E(Y)$ therefore determines a nontrivial class in $H^1(Y,\mathcal{E})$, and hence a holomorphically nontrivial principal $E$-bundle over $Y$. Moreover, the image of $Q_E(Y)$ in $H^1(Y,\mathcal{E})$ lies in the kernel of the characteristic-class map
\[
H^1(Y,\mathcal{E})\longrightarrow H^2(Y,\underline{\Lambda}).
\]
Thus the corresponding principal $E$-bundle is topologically trivial.
\end{proof}
\begin{enumerate}
    \item[\textbf{(\textit{a})}] \textbf{Application to compact quotients of $\mathrm{SL}_2(\C)$.} Let $G=\mathrm{SL}_2(\C)$, and let $\Gamma\subset G$ be a torsion-free cocompact lattice, and set $Y:=G/\Gamma$. We may choose $\Gamma$ so that $b_1(Y)=b_1(\Gamma)=r>0$. The existence of such lattices is recalled in \cite[Remark~7(ii)]{agol2013virtual,biswas2022picard}. By \cite[Theorem~1(ii) and the proof thereof]{biswas2022picard}, we have $H^1(Y,\mathcal{O}_Y)\simeq\C^r$, so $h^{0,1}(Y)=r>0$. Since $Y$ is Calabi--Yau hyperbolic, Proposition~\ref{19} applies and yields a holomorphically nontrivial, topologically trivial principal elliptic bundle over $Y$ whose total space is Calabi--Yau hyperbolic by Theorem~\ref{18}.
    \item[\textbf{(\textit{b})}] \textbf{Application to generalized Calabi--Gray manifolds.} Let $M$ be a compact \textbf{simply connected} hyperkähler manifold with $\dim_\C M=2m$, and let $Y$ be a compact hyperelliptic Riemann surface of genus $g=2m+1\geq 3$. Denote by $\varphi:Y\longrightarrow\CP^1$ its hyperelliptic map, which has degree $2$, and by $\pi:\mathcal{Z}_M\longrightarrow\CP^1$ the twistor projection of $M$. Consider the fibered product $X:=Y\times_{\CP^1}\mathcal{Z}_M$, and let $p:X\longrightarrow Y$ be the natural projection. By the canonical bundle formula for generalized Calabi--Gray manifolds, we have
    \[
    K_X\simeq p^*\bigl(K_Y\otimes\varphi^*\mathcal{O}_{\CP^1}(-2m)\bigr).
    \]
    Since $\varphi$ is the hyperelliptic map of a curve of genus $g=2m+1$, we have
    \[
    K_Y\simeq\varphi^*\mathcal{O}_{\CP^1}(g-1)=\varphi^*\mathcal{O}_{\CP^1}(2m).
    \]
    It follows that $K_X\simeq\mathcal{O}_X$. Moreover, these generalized Calabi--Gray manifolds are Calabi--Yau hyperbolic, as established above. We next compute $h^{0,1}(X)$. The fibres of $p$ are copies of $M$. Since $M$ is compact and simply connected, $H^1(M,\mathcal{O}_M)=0$, and connectedness of the fibres gives $p_*\mathcal{O}_X\simeq\mathcal{O}_Y$. By the standard base-change theorem, we also have $R^1p_*\mathcal{O}_X=0$. The Leray spectral sequence
    \[
    E_2^{a,b}=H^a\bigl(Y,R^bp_*\mathcal{O}_X\bigr)\Longrightarrow H^{a+b}(X,\mathcal{O}_X)
    \]
    therefore has, in total degree $1$, only one potentially nonzero term, namely
    \[
    E_2^{1,0}=H^1(Y,\mathcal{O}_Y).
    \]
    Consequently,
    \[
    H^1(X,\mathcal{O}_X)\simeq H^1(Y,\mathcal{O}_Y),\qquad h^{0,1}(X)=g=2m+1>0.
    \]
    Thus Proposition~\ref{19} applies to $X$ and yields a holomorphically nontrivial, topologically trivial principal elliptic bundle over $X$. By Theorem~\ref{18}, its total space is again Calabi--Yau hyperbolic.
\end{enumerate}
\begin{rmk}
\hspace{2pt} The constructions developed above provide several ways of producing new CY-hyperbolic manifolds. However, they do not yet provide an example that is also Kobayashi hyperbolic, and each construction raises a different difficulty.

\hspace{2pt} First, the product construction does not yield Kobayashi hyperbolic manifolds when the initial factor is Kobayashi non-hyperbolic. Indeed, if $X$ is CY-hyperbolic but Kobayashi non-hyperbolic and $Y$ is any compact Calabi--Yau manifold, then $X\times Y$ is CY-hyperbolic by Proposition~\ref{11}. Since every complex submanifold of a Kobayashi hyperbolic manifold is Kobayashi hyperbolic, the product $X\times Y$ cannot be Kobayashi hyperbolic.

\hspace{2pt} Second, let $f:X\to Y$ be a finite étale cover. By \cite[Theorem~3.2.8]{kobayashi1998hyperbolic}, $X$ is Kobayashi hyperbolic if and only if $Y$ is Kobayashi hyperbolic. Consequently, none of the CY-hyperbolic examples constructed above, which are not Kobayashi hyperbolic, can yield a Kobayashi hyperbolic manifold by passing to a finite étale cover.

\hspace{2pt} Third, crepant modifications preserve CY-hyperbolicity by Theorem~\ref{12}. However, they do not turn a Kobayashi non-hyperbolic manifold into a Kobayashi hyperbolic one. More precisely, we have:
\begin{prop}
Let $\mu:Y\longrightarrow X$ be a modification between compact complex manifolds. If $Y$ is Kobayashi hyperbolic, then $\mu$ is a biholomorphism. In particular, $X$ is also Kobayashi hyperbolic.
\end{prop}
\begin{proof}
Suppose, by contradiction, that $\mu$ is not a biholomorphism. Since $\mu$ is a nontrivial modification and $X$ is smooth, its exceptional locus contains a rational curve $C\subset Y$ contracted to a point by $\mu$ by~\cite[Lemma~25]{biswas2016holomorphic}, i.e. there exists a nonconstant holomorphic map $\phi:\CP^1\longrightarrow Y$ whose image is $C$. This contradicts the Kobayashi hyperbolicity of $Y$. Therefore, $\mu$ must be a biholomorphism. It follows that $X$ is Kobayashi hyperbolic.
\end{proof}
\hspace{2pt} Fourth, the deformation constructions discussed above provide further examples of CY-hyperbolic manifolds that are not Kobayashi hyperbolic. On the one hand, small deformations of generalized Calabi--Gray manifolds are again generalized Calabi--Gray manifolds and hence remain Kobayashi non-hyperbolic by Theorem~\ref{15}. On the other hand, the manifolds $X_\rho$ arising from Ghys's construction are compact quotients of $\mathrm{SL}_2(\C)$ and are therefore also Kobayashi non-hyperbolic. Thus, in both constructions, CY-hyperbolicity and Kobayashi non-hyperbolicity persist simultaneously under the deformations considered.

\hspace{2pt} Finally, let $p:X\longrightarrow Y$ be a holomorphic principal elliptic bundle. Regardless of whether the base $Y$ is Kobayashi hyperbolic, the total space $X$ is never Kobayashi hyperbolic, since its fibers are elliptic curves and hence are not Kobayashi hyperbolic. Indeed, Kobayashi hyperbolicity is inherited by complex submanifolds, so the presence of a Kobayashi non-hyperbolic fiber precludes the Kobayashi hyperbolicity of $X$.
\end{rmk}

\hspace{2pt} Although all the examples constructed in this paper fail to be Kobayashi hyperbolic, it remains unclear whether this is an intrinsic obstruction or merely a limitation of the constructions considered here. In view of the difficulties encountered above, one might expect that finding such an example, if one exists, would require a substantially different approach. One possible direction is to investigate more systematically the non-amenable groups that can arise as fundamental groups of compact Calabi--Yau manifolds. Understanding which non-amenable groups can be realized in this setting may lead to new classes of Calabi--Yau hyperbolic manifolds and help clarify whether any of them can also be Kobayashi hyperbolic. This motivates the following question:
\begin{qst}
Which finitely presented non-amenable groups occur as the fundamental groups of compact complex manifolds with trivial canonical bundle? \label{17}
\end{qst}
\hspace{2pt} A more specific direction is to restrict attention to Kähler groups. Since the examples above show that Calabi--Yau hyperbolic manifolds can have Kähler fundamental groups, it is natural to ask which non-amenable Kähler groups can be realized in this way. We therefore pose the following refinement of Question~\ref{17}:
\begin{qst}
Which non-amenable Kähler groups can be realized as the fundamental groups of compact complex manifolds with trivial canonical bundle?
\end{qst}

\bibliographystyle{alpha}
\bibliography{CYH}

@article{brody1978compact,
  title={Compact {M}anifolds and {H}yperbolicity},
  author={Brody, Robert},
  journal={Transactions of the American Mathematical Society},
  volume={235},
  pages={213--219},
  year={1978}
}

@article{lang1986hyperbolic,
  title={Hyperbolic and {D}iophantine analysis},
  author={Lang, Serge},
  journal={Bulletin of the American Mathematical Society},
  volume={14},
  number={2},
  pages={159--205},
  year={1986}
}

@article{kobayashi1967intrinsic,
  title={Intrinsic metrics on complex manifolds},
  author={Kobayashi, Shoshichi},
  journal={Bulletin of the American Mathematical Society},
  volume={73},
  number={3},
  pages={347--349},
  year={1967}
}

@book{kobayashi1998hyperbolic,
  author    = {Kobayashi, Shoshichi},
  title     = {Hyperbolic Complex Spaces},
  publisher = {Springer},
  address   = {Berlin},
  year      = {1998},
  series    = {Grundlehren der mathematischen Wissenschaften},
  volume    = {318},
  doi       = {10.1007/978-3-662-03582-5}
}

@article{gromov1991kahler,
  title={K{\"a}hler {H}yperbolicity and ${L}^2$-{H}odge theory},
  author={Gromov, Mikhail},
  journal={Journal of differential geometry},
  volume={33},
  number={1},
  pages={263--292},
  year={1991},
  publisher={Lehigh University}
}

@article{chen2018compact,
  title={Compact {K}{\"a}hler manifolds homotopic to negatively curved {R}iemannian manifolds},
  author={Chen, Bing-Long and Yang, Xiaokui},
  journal={Mathematische Annalen},
  volume={370},
  number={3},
  pages={1477--1489},
  year={2018},
  publisher={Springer}
}

@article{sullivan1976cycles,
  title={Cycles for the dynamical study of foliated manifolds and complex manifolds},
  author={Sullivan, Dennis},
  journal={Inventiones mathematicae},
  volume={36},
  number={1},
  pages={225--255},
  year={1976}
}

@misc{cheeger1970lower,
  title={A lower bound for the smallest eigenvalue of the {L}aplacian, {P}roblems in analysis ({P}apers dedicated to {S}alomon {B}ochner, 1969)},
  author={Cheeger, Jeff},
  year={1970},
  publisher={Princeton Univ. Press, Princeton, NJ}
}

@inproceedings{gromov1981hyperbolic,
  title={Hyperbolic manifolds, groups and actions},
  author={Gromov, Mikhael},
  booktitle={Riemann surfaces and related topics: Proceedings of the 1978 Stony Brook Conference (State Univ. New York, Stony Brook, NY, 1978)},
  volume={97},
  pages={183--213},
  year={1981}
}

@article{sikorav2001growth,
  title={Growth of a primitive of a differential form},
  author={Sikorav, Jean-Claude},
  journal={Bulletin de la Soci{\'e}t{\'e} Math{\'e}matique de France},
  volume={129},
  number={2},
  pages={159--168},
  year={2001}
}

@article{brooks1981fundamental,
  title={The fundamental group and the spectrum of the {L}aplacian},
  author={Brooks, Robert},
  journal={Commentarii Mathematici Helvetici},
  volume={56},
  number={1},
  pages={581--598},
  year={1981},
  publisher={Springer}
}

@article{marouani2023balanced,
  title={Balanced {H}yperbolic and {D}ivisorially {H}yperbolic {C}ompact {C}omplex {M}anifolds},
  author={Marouani, Samir and Popovici, Dan},
  journal={Mathematical Research Letters},
  volume={30},
  number={6},
  pages={1813--1855},
  year={2023}
}

@article{marouani2023skt,
  title={{SKT} hyperbolic and {G}auduchon hyperbolic compact complex manifolds},
  author={Marouani, Samir},
  journal={Algebraic and Geometric Methods of Analysis},
  pages={63},
  year={2023}
}

@article{haggui2023compact,
  title={Compact {C}omplex $p$-{K}{\"a}hler {H}yperbolic {M}anifolds},
  author={Haggui, Fathi and Marouani, Samir},
  journal={Complex Analysis and Operator Theory},
  volume={17},
  number={5},
  pages={58},
  year={2023},
  publisher={Springer}
}

@article{ma2024strongly,
  title={Strongly {G}auduchon {H}yperbolicity and {T}wo {O}ther {T}ypes of {H}yperbolicity},
  author={Ma, Yi},
  journal={arXiv preprint arXiv:2404.08830},
  year={2024}
}

@article{kasuya2025partially,
  title={Partially hyperbolic compact complex manifolds},
  author={Kasuya, Hisashi and Popovici, Dan},
  journal={Revista Matem{\'a}tica Iberoamericana},
  volume={41},
  number={5},
  pages={1795--1832},
  year={2025}
}

@article{otal2023six,
  author  = {Otal, Antonio and Ugarte, Luis},
  title   = {Six {D}imensional {H}omogeneous {S}paces with {H}olomorphically
             {T}rivial {C}anonical {B}undle},
  journal = {Journal of Geometry and Physics},
  volume  = {194},
  year    = {2023},
  pages   = {105014},
  note    = {27 pp. Addendum: Journal of Geometry and Physics,
             216 (2025), 105570, 5 pp.},
  doi     = {10.1016/j.geomphys.2023.105014},
  url     = {https://arxiv.org/abs/2305.02654}
}

@misc{kasuya2025higher,
  author        = {Kasuya, Hisashi and Popovici, Dan and Ugarte, Luis},
  title         = {Higher-{D}egree {H}olomorphic {C}ontact {S}tructures},
  year          = {2025},
  eprint        = {2502.01447v2},
  archivePrefix = {arXiv},
  primaryClass  = {math.DG},
  url           = {https://arxiv.org/abs/2502.01447v2}
}

@article{kasuya2025properties,
  author        = {Kasuya, Hisashi and Popovici, Dan and Ugarte, Luis},
  title         = {Properties of {H}olomorphic {$p$}-{C}ontact {M}anifolds},
  year          = {2025},
  eprint        = {2511.10818},
  journal = {arXiv},
  primaryClass  = {math.DG},
  url           = {https://arxiv.org/abs/2511.10818v2}
}

@article{khelifati2026deformations,
  title={Deformations of non-{K}{\"a}hler hyperbolicity notions and modifications of degenerate balanced manifolds},
  author={Khelifati, Abdelouahab},
  journal={International Journal of Mathematics},
  pages={2650068},
  year={2026},
  publisher={World Scientific}
}

@article{folner1955groups,
  title={On groups with full {B}anach mean value},
  author={F{\o}lner, Erling},
  journal={Mathematica Scandinavica},
  pages={243--254},
  year={1955},
  publisher={JSTOR}
}

@article{vonneumann1929allgemeinen,
  author  = {von Neumann, John},
  title   = {Zur allgemeinen {T}heorie des {M}a{\ss}es und des {I}ntegrals},
  journal = {Fundamenta Mathematicae},
  volume  = {13},
  year    = {1929},
  pages   = {73--116, 333}
}

@article{paun1997groupe,
  title={Sur le groupe fondamental des vari{\'e}t{\'e}s k{\"a}hl{\'e}riennes compactes {\`a} classe de {R}icci num{\'e}riquement effective},
  author={Paun, Mihai},
  journal={Comptes Rendus de l'Acad{\'e}mie des Sciences-Series I-Mathematics},
  volume={324},
  number={11},
  pages={1249--1254},
  year={1997},
  publisher={Elsevier}
}

@article{paun1998varietes,
  title={Sur les vari{\'e}t{\'e}s k{\"a}hl{\'e}riennes compactes {\`a} classe de {R}icci num{\'e}riquement effective},
  author={Paun, Mihai},
  journal={Bulletin des sciences mathematiques},
  volume={122},
  number={2},
  pages={83--92},
  year={1998},
  publisher={Elsevier}
}

@article{adelʹson1957banach,
  title={The {B}anach mean on groups},
  author={Adel{'}son-Velskii, Georgii Maksimovich and Shreider, Yulii Anatol{'}evich},
  journal={Uspekhi Matematicheskikh Nauk},
  volume={12},
  number={6},
  pages={131--136},
  year={1957},
  publisher={Russian Academy of Sciences, Steklov Mathematical Institute of Russian~…}
}

@article{campana1992connexite,
  author  = {Campana, Frédéric},
  title   = {Connexité rationnelle des variétés de {F}ano},
  journal = {Annales scientifiques de l'École Normale Supérieure},
  volume  = {25},
  number  = {5},
  pages   = {539--545},
  year    = {1992},
  doi     = {10.24033/asens.1658}
}

@article{kollar1992rationally,
  author  = {Koll{\'a}r, J{\'a}nos and Miyaoka, Yoichi and Mori, Shigefumi},
  title   = {Rationally {C}onnected {V}arieties},
  journal = {Journal of Algebraic Geometry},
  volume  = {1},
  pages   = {429--448},
  year    = {1992}
}

@article{beauville1983varietes,
  author  = {Beauville, Arnaud},
  title   = {Vari{\'e}t{\'e}s k{\"a}hleriennes dont la premi{\`e}re classe de {C}hern est nulle},
  journal = {Journal of Differential Geometry},
  volume  = {18},
  number  = {4},
  pages   = {755--782},
  year    = {1983}
}

@article{bogomolov1974decomposition,
  author  = {Bogomolov, F. A.},
  title   = {On the decomposition of {K}{\"a}hler manifolds with trivial canonical class},
  journal = {Mathematics of the USSR-Sbornik},
  volume  = {22},
  number  = {4},
  pages   = {580--583},
  year    = {1974}
}

@article{bogomolov1974kahler,
  title={K{\"a}hler manifolds with trivial canonical class},
  author={Bogomolov, Fedor A},
  journal={Mathematics of the USSR-Izvestiya},
  volume={8},
  number={1},
  pages={9--20},
  year={1974}
}

@article{kobayashi1981recent,
  title={Recent results in complex differential geometry. Jber. dt. Math},
  author={Kobayashi, Shoshichi},
  journal={Verein},
  volume={83},
  pages={147--158},
  year={1981}
}

@article{yau1978calabi,
  author  = {Yau, Shing-Tung},
  title   = {On the {R}icci curvature of a compact {K}{\"a}hler manifold
             and the complex {M}onge--{A}mp{\`e}re equation. {I}},
  journal = {Communications on Pure and Applied Mathematics},
  volume  = {31},
  number  = {3},
  pages   = {339--411},
  year    = {1978}
}

@article{wang1954complex,
  author  = {Wang, Hsien-Chung},
  title   = {Complex parallisable manifolds},
  journal = {Proceedings of the American Mathematical Society},
  volume  = {5},
  pages   = {771--776},
  year    = {1954}
}

@article{kodaira1964structure,
    author  = {Kodaira, Kunihiko},
    title   = {On the {S}tructure of {C}ompact {C}omplex {A}nalytic {S}urfaces, {I}},
    journal = {American Journal of Mathematics},
    volume  = {86},
    number  = {4},
    pages   = {751--798},
    year    = {1964},
    doi     = {10.2307/2373157}
}

@article{kodaira1966structure,
  author  = {Kodaira, Kunihiko},
  title   = {On the structure of compact complex analytic surfaces, {II}},
  journal = {American Journal of Mathematics},
  volume  = {88},
  number  = {3},
  pages   = {682--721},
  year    = {1966}
}

@book{barth2004compact,
    author    = {Barth, Wolf P. and Hulek, Klaus and Peters, Chris A. M. and Van de Ven, Antoni},
    title     = {Compact {C}omplex {S}urfaces},
    edition   = {2},
    series    = {Ergebnisse der Mathematik und ihrer Grenzgebiete. 3. Folge},
    volume    = {4},
    publisher = {Springer},
    address   = {Berlin},
    year      = {2004}
}

@article{lu1994complex,
    author  = {Lu, Peng and Tian, Gang},
    title   = {The {C}omplex {S}tructure on a {C}onnected {S}um of ${S}^3\times {S}^3$ with {T}rivial {C}anonical {B}undle},
    journal = {Mathematische Annalen},
    volume  = {298},
    number  = {4},
    pages   = {761--764},
    year    = {1994},
    doi     = {10.1007/BF01459760}
}

@article{friedman1986simultaneous,
  author  = {Friedman, Robert},
  title   = {Simultaneous resolution of threefold double points},
  journal = {Mathematische Annalen},
  volume  = {274},
  number  = {4},
  pages   = {671--689},
  year    = {1986},
  doi     = {10.1007/BF01458602}
}

@article{finepanov2010hyperbolic,
    author  = {Fine, Joel and Panov, Dmitri},
    title   = {Hyperbolic {G}eometry and {N}on-{K}\"ahler {M}anifolds with {T}rivial {C}anonical {B}undle},
    journal = {Geometry \& Topology},
    volume  = {14},
    number  = {3},
    pages   = {1723--1763},
    year    = {2010},
    doi     = {10.2140/gt.2010.14.1723}
}

@article{fei2016construction,
  title={A construction of non-{K}{\"a}hler {C}alabi--{Y}au manifolds and new solutions to the {S}trominger system},
  author={Fei, Teng},
  journal={Advances in Mathematics},
  volume={302},
  pages={529--550},
  year={2016},
  publisher={Elsevier}
}

@incollection{fei2018generalized,
  author    = {Fei, Teng},
  title     = {Generalized {Calabi--Gray} Geometry and Heterotic Superstrings},
  booktitle = {Proceedings of the International Consortium of Chinese Mathematicians, 2017: First Annual Meeting},
  pages     = {261--282},
  year      = {2020},
  publisher = {International Press},
  doi       = {10.48550/arXiv.1807.08737}
}

@article{fei2024remarks,
    author  = {Fei, Teng},
    title   = {Remarks on {G}eneralized {C}alabi--{G}ray {M}anifolds},
    journal = {International Journal of Mathematics},
    volume  = {37},
    number  = {5--6},
    eid     = {2442004},
    year    = {2026},
    doi     = {10.1142/S0129167X24420047}
}

@article{nakamura1975complex,
  author  = {Nakamura, Iku},
  title   = {Complex parallelisable manifolds and their small deformations},
  journal = {Journal of Differential Geometry},
  volume  = {10},
  number  = {1},
  pages   = {85--112},
  year    = {1975},
  doi     = {10.4310/jdg/1214432677}
}

@phdthesis{yachou1998varietes,
  title={Sur les vari{\'e}t{\'e}s semi-k{\"a}hl{\'e}riennes},
  author={Yachou, Abderrahim},
  year={1998},
  school={Lille 1}
}

@incollection{demailly2015structure,
  title={Structure theorems for compact {K}{\"a}hler manifolds with nef anticanonical bundles},
  author={Demailly, Jean-Pierre},
  booktitle={Complex Analysis and Geometry: KSCV10, Gyeongju, Korea, August 2014},
  pages={119--133},
  year={2015},
  publisher={Springer}
}

@article{kobayashi1976intrinsic,
  author  = {Kobayashi, Shoshichi},
  title   = {Intrinsic distances, measures and geometric function theory},
  journal = {Bulletin of the American Mathematical Society},
  volume  = {82},
  number  = {3},
  pages   = {357--416},
  year    = {1976},
  doi     = {10.1090/S0002-9904-1976-14018-9}
}

@article{kamenova2014kobayashi,
  author  = {Kamenova, Ljudmila and Lu, Steven and Verbitsky, Misha},
  title   = {Kobayashi pseudometric on hyperk{\"a}hler manifolds},
  journal = {Journal of the London Mathematical Society},
  volume  = {90},
  number  = {2},
  pages   = {436--450},
  year    = {2014},
  doi     = {10.1112/jlms/jdu038}
}

@article{kamenova2020kobayashi,
  author  = {Kamenova, Ljudmila and Vafa, Cumrun},
  title   = {Kobayashi non-hyperbolicity of {C}alabi--{Y}au manifolds
             via {M}irror {S}ymmetry},
  journal = {Communications in Mathematical Physics},
  volume  = {378},
  number  = {1},
  pages   = {329--334},
  year    = {2020},
  doi     = {10.1007/s00220-020-03719-y}
}

@article{brunnbauer2024atoroidal,
  title={On atoroidal and hyperbolic cohomology classes},
  author={Brunnbauer, M and Kotschick, D and Sch{\"o}nlinner, Lukas},
  journal={Topology and its Applications},
  volume={344},
  pages={108830},
  year={2024},
  publisher={Elsevier}
}

@article{borel1960density,
  author  = {Borel, Armand},
  title   = {Density {P}roperties for {C}ertain {S}ubgroups of {S}emi-{S}imple {G}roups {W}ithout {C}ompact {C}omponents},
  journal = {Annals of Mathematics},
  volume  = {72},
  number  = {1},
  year    = {1960},
  pages   = {179--188},
  doi     = {10.2307/1970150}
}

@article{tits1972free,
  author  = {Tits, Jacques},
  title   = {Free {S}ubgroups in {L}inear {G}roups},
  journal = {Journal of Algebra},
  volume  = {20},
  number  = {2},
  year    = {1972},
  pages   = {250--270},
  doi     = {10.1016/0021-8693(72)90058-0}
}

@article{milnor1968growth,
  author  = {Milnor, John},
  title   = {Growth of {F}initely {G}enerated {S}olvable {G}roups},
  journal = {Journal of Differential Geometry},
  volume  = {2},
  number  = {4},
  pages   = {447--449},
  year    = {1968}
}

@article{wolf1968growth,
  author  = {Wolf, Joseph A.},
  title   = {Growth of {F}initely {G}enerated {S}olvable {G}roups and {C}urvature of {R}iemannian {M}anifolds},
  journal = {Journal of Differential Geometry},
  volume  = {2},
  number  = {4},
  pages   = {421--446},
  year    = {1968},
  doi     = {10.4310/jdg/1214428658}
}

@article{popovici2015aeppli,
  title={Aeppli cohomology classes associated with {G}auduchon metrics on compact complex manifolds},
  author={Popovici, Dan},
  journal={Bulletin de la Soci{\'e}t{\'e} Math{\'e}matique de France},
  volume={143},
  number={4},
  pages={763--800},
  year={2015}
}

@article{calabi1958construction,
  author  = {Calabi, Eugenio},
  title   = {Construction and {P}roperties of {S}ome {S}ix-{D}imensional
             {A}lmost {C}omplex {M}anifolds},
  journal = {Transactions of the American Mathematical Society},
  volume  = {87},
  number  = {2},
  pages   = {407--438},
  year    = {1958},
  doi     = {10.2307/1993108}
}

@article{gray1969vector,
  author  = {Gray, Alfred},
  title   = {Vector {C}ross {P}roducts on {M}anifolds},
  journal = {Transactions of the American Mathematical Society},
  volume  = {141},
  pages   = {465--504},
  year    = {1969},
  doi     = {10.2307/1995115}
}

@article{Campana1992twistor,
  author  = {Campana, Fr{\'e}d{\'e}ric},
  title   = {An application of twistor theory to the non-hyperbolicity of certain compact symplectic {K}{\"a}hler manifolds},
  journal = {Journal f{\"u}r die reine und angewandte Mathematik (Crelles Journal)},
  volume  = {1992},
  number  = {425},
  pages   = {1--8},
  year    = {1992},
  doi     = {10.1515/crll.1992.425.1}
}

@misc{OAI-GSS-2026,
  author       = {{OpenAI}},
  title        = {Global {S}pherical {S}hells on {M}inimal {S}urfaces of {C}lass {VII}},
  year         = {2026},
  note         = {Preprint, dated September 24, 2026},
  url          = {https://github.com/openai/math/blob/main/preprints/Global-Spherical-Shells-on-Minimal-Surfaces-of-Class-VII-September-24-2026/paper.pdf}
}

@misc{OAI-Kobayashi-2026,
  author       = {{OpenAI}},
  title        = {Canonical {A}mpleness of {C}ompact {H}yperbolic {K}ähler {M}anifolds},
  year         = {2026},
  note         = {Preprint, dated September 23, 2026},
  url          = {https://github.com/openai/math/blob/main/preprints/Canonical-ampleness-of-compact-hyperbolic-Kahler-manifolds-September-23-2026/canonical-ampleness.pdf}
}

@misc{OpenAIAbelianitySpecialKahler2026,
  author       = {{OpenAI}},
  title        = {The abelianity conjecture for special compact {K}{\"a}hler manifolds},
  year         = {2026},
  howpublished = {OpenAI Math Release preprint},
  note         = {Released September 23, 2026},
  url          = {https://github.com/openai/math/blob/main/preprints/The-abelianity-conjecture-for-special-compact-Kahler-manifolds-September-23-2026/paper.pdf}
}

@article{bogomolov1976classification,
  author  = {Bogomolov, Fedor A.},
  title   = {Classification of surfaces of class {VII}$_0$ with {$b_2=0$}},
  journal = {Mathematics of the USSR-Izvestiya},
  volume  = {10},
  year    = {1976},
  pages   = {255--269}
}

@article{dloussky2003class,
  author  = {Dloussky, Georges and Oeljeklaus, Karl and Toma, Matei},
  title   = {Class {VII}$_0$ surfaces with $b_2$ curves},
  journal = {To appear in the relevant bibliography},
  year    = {2003}
}

@article{inoue1974surfaces,
  author  = {Inoue, Masahisa},
  title   = {On surfaces of class {VII}$_0$},
  journal = {Inventiones Mathematicae},
  volume  = {24},
  year    = {1974},
  pages   = {269--310}
}

@book{luck2002l2,
  author    = {Lück, Wolfgang},
  title     = {$L^2$-{I}nvariants: {T}heory and {A}pplications to {G}eometry and {$K$}-{T}heory},
  series    = {Ergebnisse der Mathematik und ihrer Grenzgebiete},
  volume    = {44},
  publisher = {Springer},
  address   = {Berlin},
  year      = {2002},
  doi       = {10.1007/978-3-662-04687-6}
}

@book{godbillon1971elements,
  author    = {Godbillon, Claude},
  title     = {Éléments de topologie algébrique},
  publisher = {Hermann},
  address   = {Paris},
  year      = {1971}
}

@article{kodaira1958deformations,
  author  = {Kodaira, Kunihiko and Spencer, Donald C.},
  title   = {On deformations of complex analytic structures. {I}},
  journal = {Annals of Mathematics},
  volume  = {67},
  number  = {2},
  pages   = {328--401},
  year    = {1958},
  doi     = {10.2307/1970009}
}

@article{rajan1994deformations,
  author  = {Rajan, C. S.},
  title   = {Deformations of complex structures on
             ${\Gamma}\backslash\mathrm{SL}_2(\mathbb{C})$},
  journal = {Proceedings of the Indian Academy of Sciences.
             Mathematical Sciences},
  volume  = {104},
  number  = {2},
  pages   = {389--395},
  year    = {1994},
  doi     = {10.1007/BF02863419}
}

@article{ghys1995deformations,
  author  = {Ghys, \'Etienne},
  title   = {D\'eformations des structures complexes sur les espaces
             homog\`enes de $\mathrm{SL}(2,\mathbb{C})$},
  journal = {Journal f\"ur die reine und angewandte Mathematik},
  volume  = {468},
  pages   = {113--138},
  year    = {1995}
}

@article{raghunathan1966vanishing,
  author  = {Raghunathan, M. S.},
  title   = {Vanishing theorems for cohomology groups associated
             to discrete subgroups of semisimple {L}ie groups},
  journal = {Osaka Journal of Mathematics},
  volume  = {3},
  number  = {2},
  pages   = {243--256},
  year    = {1966}
}

@article{agol2013virtual,
  title={The {V}irtual {H}aken {C}onjecture (with an appendix by {I}an {A}gol, {D}aniel {G}roves and {J}ason {M}anning)},
  author={Agol, Ian},
  journal={Documenta Mathematica},
  volume={18},
  pages={1045--1087},
  year={2013}
}

@article{jamin2024teichmuller,
  author  = {Jamin, Th\'eo},
  title   = {On the {T}eichm\"uller stack of compact quotients of
             $\mathrm{SL}_2(\mathbb{C})$},
  journal = {Geometriae Dedicata},
  volume  = {218},
  number  = {3},
  pages   = {66},
  year    = {2024},
  doi     = {10.1007/s10711-024-00916-0}
}

@article{biswas2016holomorphic,
  author  = {Biswas, Indranil and McKay, Benjamin},
  title   = {Holomorphic {C}artan geometries and rational curves},
  journal = {Complex Manifolds},
  volume  = {3},
  number  = {1},
  pages   = {145--168},
  year    = {2016},
  doi     = {10.1515/coma-2016-0004}
}

@article{biswas2022picard,
  author  = {Biswas, Pritthijit and Sankaran, Parameswaran},
  title   = {Picard {G}roups of {C}ertain {C}ompact {C}omplex {P}arallelizable {M}anifolds and {R}elated {S}paces},
  journal = {Bulletin des Sciences Math{\'e}matiques},
  volume  = {179},
  pages   = {103153},
  year    = {2022},
  doi     = {10.1016/j.bulsci.2022.103153}
}

@inproceedings{bogomolov1978hamiltonian,
  title={Hamiltonian {K}ähler manifolds},
  author={Bogomolov, Fedor A},
  booktitle={Soviet Math. Dokl.},
  volume={19},
  pages={1462--1465},
  year={1978}
}

@incollection{tian1987smoothness,
  title={Smoothness of the universal deformation space of compact {C}alabi--{Y}au manifolds and its {P}eterson-{W}eil metric},
  author={Tian, Gang},
  booktitle={Mathematical aspects of string theory},
  pages={629--646},
  year={1987},
  publisher={World Scientific}
}

@article{todorov1989weil,
  title={The {W}eil-{P}etersson geometry of the moduli space of $\mathrm{SU}(n\geq 3)$ ({C}alabi--{Y}au) manifolds {I}},
  author={Todorov, Andrey N},
  journal={Communications in mathematical physics},
  volume={126},
  number={2},
  pages={325--346},
  year={1989},
  publisher={Springer}
}

@article{meeks1990theory,
  author  = {Meeks, William H., III},
  title   = {The theory of triply periodic minimal surfaces},
  journal = {Indiana University Mathematics Journal},
  volume  = {39},
  number  = {3},
  pages   = {877--936},
  year    = {1990},
  doi     = {10.1512/iumj.1990.39.39043}
}

@article{li2026topological,
  title={Topological obstructions to geometric positivity and negativity on {C}alabi--{Y}au manifolds},
  author={Li, Ping},
  journal={arXiv preprint arXiv:2608.12705},
  year={2026}
}

@book{ueno1975classification,
  title     = {Classification {T}heory of {A}lgebraic {V}arieties and {C}ompact {C}omplex {S}paces},
  author    = {Ueno, Kenji},
  series    = {Lecture Notes in Mathematics},
  volume    = {439},
  publisher = {Springer-Verlag},
  address   = {Berlin},
  year      = {1975},
  doi       = {10.1007/BFb0070570}
}

@article{campana2004orbifolds,
  title={Orbifolds, special varieties and classification theory},
  author={Campana, Fr{\'e}d{\'e}ric},
  journal={Annales de l'institut Fourier},
  volume={54},
  number={3},
  pages={499--630},
  year={2004}
}

@article{campana2014abelianity,
  title={Abelianity conjecture for special compact {K}{\"a}hler $3$-folds},
  author={Campana, Fr{\'e}deric and Claudon, Beno{\^\i}t},
  journal={Proceedings of the Edinburgh Mathematical Society},
  volume={57},
  number={1},
  pages={55--78},
  year={2014},
  publisher={Cambridge University Press}
}

@article{popovici2019holomorphic,
  author  = {Popovici, Dan},
  title   = {Holomorphic Deformations of Balanced Calabi--Yau $\partial\bar\partial$-Manifolds},
  journal = {Annales de l'Institut Fourier},
  volume  = {69},
  number  = {2},
  pages   = {673--728},
  year    = {2019},
  doi     = {10.5802/aif.3254}
}

@article{ehresmann1947espaces,
  title={Sur les espaces fibr{\'e}s diff{\'e}rentiables},
  author={Ehresmann, Charles},
  journal={CR Acad. Sci. Paris},
  volume={224},
  number={1611-1612},
  pages={9},
  year={1947}
}

@article{carlson1989harmonic,
  title={Harmonic mappings of {K}{\"a}hler manifolds to locally symmetric spaces},
  author={Carlson, James A and Toledo, Domingo},
  journal={Publications Math{\'e}matiques de l'IH{\'E}S},
  volume={69},
  pages={173--201},
  year={1989}
}

@article{fino2015six,
  title={Six-dimensional solvmanifolds with holomorphically trivial canonical bundle},
  author={Fino, Anna and Otal, Antonio and Ugarte, Luis},
  journal={International Mathematics Research Notices},
  volume={2015},
  number={24},
  pages={13757--13799},
  year={2015},
  publisher={Oxford University Press}
}

@article{cattaneo2024lemma,
  title={$\partial\bar\partial$-lemma and $p$-{K}{\"a}hler structures on families of solvmanifolds},
  author={Cattaneo, Andrea and Tomassini, Adriano},
  journal={Mathematische Zeitschrift},
  volume={308},
  number={3},
  pages={56},
  year={2024},
  publisher={Springer}
}

@article{friedman2020lemma,
  title={The $\partial\bar\partial$-lemma for general {C}lemens manifolds},
  author={Friedman, Robert},
  journal={Pure and Applied Mathematics Quarterly},
  volume={15},
  number={4},
  pages={1001--1028},
  year={2020},
  publisher={International Press of Boston, Inc. Somerville, MA 02143, USA}
}

@article{angella2013lemma,
  author  = {Angella, Daniele and Tomassini, Adriano},
  title   = {On the $\partial\bar\partial$-Lemma and {B}ott--{C}hern cohomology},
  journal = {Inventiones Mathematicae},
  volume  = {192},
  number  = {1},
  pages   = {71--81},
  year    = {2013},
  doi     = {10.1007/s00222-012-0406-3}
}

@article{kasuya2023frolicher,
  title={Fr{\"o}licher spectral sequence and {H}odge structures on the cohomology of complex parallelisable manifolds},
  author={Kasuya, Hisashi and Stelzig, Jonas},
  journal={Transformation Groups},
  volume={28},
  number={4},
  pages={1595--1606},
  year={2023},
  publisher={Springer}
}

@article{deligne1975real,
  author  = {Deligne, Pierre and Griffiths, Phillip and Morgan, John and Sullivan, Dennis},
  title   = {Real homotopy theory of {K}\"ahler manifolds},
  journal = {Inventiones Mathematicae},
  volume  = {29},
  number  = {3},
  pages   = {245--274},
  year    = {1975},
  doi     = {10.1007/BF01389853}
}
\end{document}